\documentclass[11pt]{amsart}

\pdfoutput=1

\usepackage{amsmath} 
\allowdisplaybreaks 
\usepackage[foot]{amsaddr}

\usepackage{color}
\usepackage{appendix}
\usepackage{cancel}
\usepackage{esint,amssymb} 
\usepackage{graphicx}
\usepackage{mathtools} 
\usepackage{xcolor}

\usepackage{microtype}

\usepackage{bm}
\usepackage{mathrsfs}
\usepackage{enumitem}
\usepackage[colorlinks=true, pdfstartview=FitV, linkcolor=blue, citecolor=blue, urlcolor=blue,pagebackref=false]{hyperref}
\usepackage{orcidlink}

\definecolor{darkgreen}{rgb}{0,0.5,0}
\definecolor{darkblue}{rgb}{0,0,0.7}
\definecolor{darkred}{rgb}{0.9,0.1,0.1}

\newtheorem{proposition}{Proposition}
\newtheorem{theorem}[proposition]{Theorem}
\newtheorem{lemma}[proposition]{Lemma}
\newtheorem{corollary}[proposition]{Corollary}

\theoremstyle{remark}
\newtheorem{remark}[proposition]{Remark}

\theoremstyle{definition}
\newtheorem{definition}[proposition]{Definition}

\numberwithin{equation}{section}
\numberwithin{proposition}{section}
\numberwithin{figure}{section}
\numberwithin{table}{section}

\newcommand{\Z}{\mathbb{Z}}
\newcommand{\N}{\mathbb{N}}
\newcommand{\Q}{\mathbb{Q}}
\newcommand{\R}{\mathbb{R}}

\newcommand{\E}{\mathbb{E}}
\renewcommand{\P}{\mathbb{P}}

\renewcommand{\leq}{\leqslant}
\renewcommand{\geq}{\geqslant}

\renewcommand{\subset}{\subseteq}
\renewcommand{\bar}{\overline}
\renewcommand{\tilde}{\widetilde}

\newcommand{\Ll}{\left}
\newcommand{\Rr}{\right}
\renewcommand{\d}{\mathrm{d}}

\newcommand{\mcl}{\mathcal}
\newcommand{\msf}{\mathsf}
\newcommand{\mfk}{\mathfrak}
\newcommand{\msc}{\mathscr}

\newcommand{\eps}{\varepsilon}

\newcommand{\la}{\left\langle}
\newcommand{\ra}{\right\rangle}

\newcommand{\Id}{\mathsf{Id}}
\newcommand{\id}{\mathsf{id}}

\renewcommand{\S}{\mathrm{S}}

\DeclareMathOperator{\supp}{supp}

\newcommand{\upa}{\uparrow}

\newcommand{\fR}{\mathfrak{R}}
\newcommand{\scdot}{\,\cdot\,} 
\newcommand{\sfr}{{\mathsf{r}}}
\newcommand{\sfp}{{\mathsf{p}}}
\newcommand{\sfq}{{\mathsf{q}}}
\newcommand{\sfs}{{\mathsf{s}}}
\newcommand{\sS}{\mathscr{S}}
\newcommand{\s}{{\mathsf{s}}}

\newenvironment{e}{\begin{equation}}{\end{equation}\ignorespacesafterend}
\newenvironment{e*}{\begin{equation*}}{\end{equation*}\ignorespacesafterend}

\begin{document}

\author{Hong-Bin Chen\,\orcidlink{0000-0001-6412-0800}}
\address[Hong-Bin Chen]{NYU-ECNU Institute of Mathematical Sciences, NYU Shanghai, China}
\email{\href{mailto:hongbin.chen@nyu.edu}{hongbin.chen@nyu.edu}}

\author{Victor Issa\,\orcidlink{0009-0009-1304-046X}}
\address[Victor Issa]{Department of Mathematics, ENS Lyon, France}
\email{\href{mailto:victor.issa@ens-lyon.fr}{victor.issa@ens-lyon.fr}}

\subjclass[2010]{}

\title[High-temperature free energy for non-convex spin glasses]{High-temperature free energy for non-convex mean-field spin glasses with Mattis interactions}

\begin{abstract}
We study enriched mean-field vector spin glasses whose covariance function need not be convex. In a quantitative short-time regime, corresponding to sufficiently high temperature in the usual physical parameterization, we prove convergence of the quenched free energy to the differentiable solution of a infinite-dimensional Hamilton--Jacobi equation proposed for non-convex models. We identify that solution through its characteristics and a unique fixed point. The argument does not assume replica symmetry.
We then allow a general continuous Mattis interaction $N G(m_N)$, obtain the corresponding variational formula, and prove a quenched large deviation principle for $m_N$. As an explicit application, we treat a bipartite model with Mattis fields, which is closely related to the Restricted Boltzmann Machine. The application requires neither a replica-symmetry assumption nor ellipticity or Nishimori identities; its restriction is the same quantitative high-temperature condition as in the general theorem.

    \bigskip


    \medskip

                 
\end{abstract}

\maketitle

\thispagestyle{empty}
{
  \hypersetup{linkcolor=black}
  \setcounter{tocdepth}{1}
  \tableofcontents
}


\section{Introduction}

\subsection{Preamble}

We prove high-temperature, equivalently short-time, convergence of the free energy for a broad class of mean-field spin glasses without assuming convexity of the covariance function $\xi$. The limit is the differentiable solution of an infinite-dimensional Hamilton--Jacobi equation, identified by an invertible characteristic flow. This is the central point of the paper: the proof does not use Guerra interpolation and does not assume replica symmetry. 

An explicit application is the bipartite Gaussian model with Mattis fields
\begin{equation} \label{e.simple}
    E_N(\sigma) = \frac{\beta}{\sqrt{N}} \sum_{i,j = 1}^N W_{ij} \sigma_{1,i} \sigma_{2,j} + \sum_{i = 1}^N \chi^1_i \sigma_{1,i} +  \sum_{i = 1}^N \chi^2_i \sigma_{2,i}, 
\end{equation}
where $\sigma=(\sigma_1,\sigma_2)\in\{-1,1\}^N\times\{-1,1\}^N$, the $(W_{ij})$ are independent standard Gaussian variables, and the two families $(\chi_i^1)$ and $(\chi_i^2)$ are independent symmetric signs. Restricted Boltzmann Machines and related bipartite spin glasses motivate this non-convex covariance structure \cite{smolensky1986harmonytheory,tubiana2018RBM,barra2018rbm}. We identify, for sufficiently small $\beta$, the limit of
\begin{equation} \label{e.F_Nbasic}
    f_N = -\frac{1}{N} \log \left( \frac{1}{4^N} \sum_{\sigma \in \{-1,1\}^N \times \{-1,1\}^N} e^{E_N(\sigma)} \right),
\end{equation}
and prove a quenched large deviation principle (Definition~\ref{d.ldp}) for the $\R^2$-valued Mattis magnetization
\begin{align*}
    m_N= (m_{N,1},m_{N,2}) = \Ll(\frac{1}{N} \sum_{i = 1}^N \chi^1_i \sigma_{1,i},\quad  \frac{1}{N} \sum_{i = 1}^N \chi^2_i \sigma_{2,i}\Rr)
\end{align*}
under the corresponding Gibbs measure $\la \cdot \ra_N$, namely
\begin{equation*}
 \la g(\sigma) \ra_N \propto  \frac{1}{4^N} \sum_{\sigma \in \{-1,1\}^N \times \{-1,1\}^N } g(\sigma) e^{E_N(\sigma)}.
\end{equation*}

Rigorous analyses of bipartite spin glasses and Boltzmann machines have obtained high-temperature controls and replica-symmetric formulas, as well as exact formulas under special structures such as the Nishimori identities; see, for example, \cite{barramulti,agliari2019free,genovese2022minimax,alberici2021solution}. Our result has a different scope. It proves free-energy convergence for the non-convex Gaussian model~\eqref{e.simple} without imposing a replica-symmetry ansatz, an ellipticity or positivity condition on the species interaction, or Nishimori identities. The restriction is instead the quantitative high-temperature interval stated below. 


For convex $\xi$, the Parisi formula is available \cite{pan.potts,pan.vec}. Without convexity the situation is radically different, Guerra's comparison fails, and even the existence of the thermodynamic limit is delicate. In that context, the Parisi formula is provably false. The Hamilton--Jacobi program of \cite{mourrat2019parisi,mourrat2020nonconvex,mourrat2020free,HJbook} predicts that the limit free energy can instead be described using the solution of an infinite-dimensional partial differential equation with a monotone càdlàg path $q$ as spatial variable. We verify that prediction on a quantitative interval $0\leq t<t_\star$. This interval is a characteristic-invertibility regime; its endpoint is not asserted to be a physical critical temperature.

The general free-energy and large-deviation statements are Theorems~\ref{t.free_energy} and~\ref{t.ldp}. The former identifies the enriched limit through a unique path-valued fixed point. 
We now state their specialization to~\eqref{e.simple}.

The application of our main results to the model in~\eqref{e.simple} can be summarized as follows.
We let $\eta$ be a standard real-valued Gaussian random variable, and $\chi$ be a centered $\pm 1$ Bernoulli random variable. For $y \in \R^2_+$ and $x \in \R^2$, we define
\begin{align*}
    \phi(y;x) =  \sum_{i=1}^2 \Ll( - \E \log \cosh(\sqrt{2y_i} \eta + x_i \chi) + y_i\Rr).
\end{align*}
%
For some $\beta_\star$ specified in the theorem below, we will show in Proposition~\ref{p.f-d-proj} that, for every $x\in\R^2$, the partial differential equation 
\begin{align*}
    \begin{cases}
        \partial_t g(t,y;x) - \partial_{y_1}g(t,y;x) \partial_{y_2}g(t,y;x)=0 \\ 
        g(0,y;x) = \phi(y;x) 
    \end{cases}
\end{align*}
admits a unique differentiable solution $g(\cdot,\scdot;x)$ on $[0,\beta^2_\star/2)\times\R^2_+$.

 \begin{theorem}\label{t.parisi+ldp.basic} 
    There exists $\beta_\star \geq \frac{1}{2\sqrt{2}}$ such that for every $\beta < \beta_\star$ and almost every realization of $(W_{i,j})_{i,j}$, $(\chi^1_i)_i$ and $(\chi^2_i)_i$, the following holds.  We have 
    \begin{e} \label{e.limf_nbasic}
        \lim_{N \to +\infty} f_N = \inf_{m \in \R^2}\sup_{x \in \R^2}\Ll\{g\Ll(\tfrac{\beta^2}{2},0;x\Rr)+m\cdot x-(m_1+m_2)\Rr\}- \tfrac{\beta^2}{2},
    \end{e}
    and the $\R^2$-valued mean magnetization $m_N$ satisfies a large deviation principle under $\langle \cdot \rangle_N$ with rate function $J$ defined for $m\in\R^2$ by
    \begin{e*}
        J(m) = -(m_1+m_2) + g^*\left( \tfrac{\beta^2}{2}, 0;m\right) + \sup_{m' \in \R^2} \left\{m'_1+m'_2 - g^*\left( \tfrac{\beta^2}{2}, 0;m'\right) \right\},
    \end{e*}
    where $ g^*\left( \tfrac{\beta^2}{2},0;m\right) = \sup_{x \in \R^2} \left\{x\cdot m  + g\left( \tfrac{\beta^2}{2},0;x\right)\right\}.$
\end{theorem}

\begin{remark} \label{r.RBM_fixed_point}
We mention that the quantity $g\Ll(\tfrac{\beta^2}{2},0;x\Rr)$ can be understood in more familiar terms using a critical point representation. For $t < \beta_*^2/2$, there is a unique solution $z^x =(z^x_1,z^x_2) \in\R^2_+$ to the fixed point equation 
\begin{align}\label{e.RBM_fixed_point}
    z=\nabla_y\phi(t(z_2,z_1);x)\quad\text{for }z\in\R^2_+,
\end{align}
and we have
\begin{align}\label{e.RBM_crit_point}
    g\Ll(t,0;x\Rr)=\phi\Ll(t(z^x_2,z^x_1);x\Rr)-tz^x_1z^x_2.
\end{align}
In addition, we also have that $\nabla_y g(t,0;x)= z^x$. \qed
\end{remark}
%

We refer to~\eqref{e.RBM_crit_point} as a critical-point representation since $z^x$ as a solution of~\eqref{e.RBM_fixed_point} is a critical point of the function on the right-hand side of~\eqref{e.RBM_crit_point}. Also, we do not expect that the lower bound $\frac{1}{2\sqrt{2}}$ for $\beta_\star$ is sharp.


Replica-symmetric formulas for bipartite and multilayer neural-network models have been obtained by interpolation, self-averaging, minimax, and transport methods \cite{barra2010rsapprox,agliari2019free,alberici2020annealing,genovese2022minimax,agliari2021transport}; exact results are also available in special structures such as the Nishimori regime \cite{alberici2021solution}. Our result addresses a different point: it identifies the enriched non-convex limit throughout a short-time characteristic regime without imposing an RS ansatz. For the Mattis part, analyses of particular models often use constrained free energies \cite{guionnet2025estimating,franz2020largedeviations,camilli2022mismatch,pan.vec,chen2014mixedSK}. We instead keep $G$ as a parameter and study
\begin{align*}
    f_N^G = -\frac{1}{N} \log \bigg( \frac{1}{4^N} &\sum_{\sigma \in \{-1,1\}^N \times \{-1,1\}^N} \\
    &\qquad\qquad\exp\bigg(\frac{\beta}{\sqrt{N}} \sum_{i,j = 1}^N W_{ij} \sigma_{1,i} \sigma_{2,j} + NG(m_N) \bigg) \bigg).
\end{align*}
without imposing a magnetization constraint.

To compute $\lim_{N \to +\infty} f_N^G$, we can proceed as follows. We let $\langle \cdot \rangle^G_N$ denote the Gibbs measure associated with $f_N^G$. We start with the case $G(m) = x \cdot m$. Under this assumption, the sum structure of  
\begin{e*}
    N G(m_N) = x_1 \sum_{i = 1}^N  \chi_i^1 \sigma_{1,i} + x_2 \sum_{i = 1}^N  \chi_i^2 \sigma_{2,i}
\end{e*}
makes the computation of $\lim_{N \to +\infty} f_N^{m \mapsto x\cdot m}$ extremely similar to the computation of $\lim_{N \to +\infty} f_N^{0}$. We show that in this case the free energy can be computed, at least in the small $\beta$ region, using a critical point representation for $\lim_{N \to +\infty} f_N^{m \mapsto x\cdot m}$ in the style of Remark~\ref{r.RBM_fixed_point}. This representation can then be connected to the aforementioned Hamilton--Jacobi equation.
Then, as a direct application of the Gärtner–Ellis theorem, we deduce from the knowledge of $\lim_{N \to +\infty} f_N^{m \mapsto x\cdot m}$ for all $x \in \R^2$, a large deviation principle for $m_N$ under $\langle \cdot \rangle_N^0$ with an explicit rate function $I$. Varadhan's lemma then implies that for every continuous function $G$,
\begin{e*}
    \lim_{N \to +\infty} \frac{1}{N} \log \left\langle e^{N G(m_N) } \right\rangle_N^{0} = \sup_{m} \left\{ G(m) - I(m) \right\}.
\end{e*}
Observing that $\frac{1}{N} \log \left\langle e^{N G(m_N) } \right\rangle_N^{0} = -f_N^G + f_N^0$, it follows that $\lim_{N \to +\infty} f_N^G$ can be computed. Bryc's inverse Varadhan lemma then gives the LDP under $\langle \cdot \rangle_N^G$. Taking $G(m)=m_1+m_2$ yields Theorem~\ref{t.parisi+ldp.basic}.

\subsection{Setting} \label{ss.setting}
In this section, we define the class of models that we treat. 

For $m,n\in\N$, we denote by $\R^{m\times n}$ the space of all $m\times n$ real matrices. For any $a=(a_{ij})_{1\leq i\leq m,\, 1\leq j\leq n}\in \R^{m\times n}$, we denote its $j$-th column vector as $a_{\bullet j}=(a_{ij})_{1\leq i\leq m}$ and its $i$-th row vector as $a_{i\bullet}=(a_{ij})_{1\leq j\leq n}$. If not specified, a vector is understood to be a column vector. For a matrix or vector $a$, we denote by $a^\intercal$ its transpose.
For $a,b\in \R^{m\times n}$, we write $a\cdot b= \sum_{ij}a_{ij}b_{ij}$, $|a|=\sqrt{a\cdot a}$, and similarly for vectors. 

For $n\in \N$, let $\S^n$ be the linear space of $n\times n$ real symmetric matrices 
and also let $\S^n_+$ (resp.\ $\S^n_{++}$) be the subset consisting of positive semi-definite (resp.\ definite) matrices. For $a,b \in \S^n$, we write $a\geq b$ provided $a-b\in\S^n_+$, which gives a natural partial order on $\S^n$.

We start by defining the spin glass part of the model. Let $D\in\N$ be the dimension of a single spin. Let $P_1$ be a finite measure supported on the unit ball $\Ll\{\tau\in\R^D:\:|\tau|\leq1\Rr\}$ of $\R^D$ and we assume that the linear span of vectors in $\supp P_1$ is $\R^D$. 

For each $N\in\N$, we denote by $\sigma =(\sigma_{\bullet i})_{i=1}^N =(\sigma_{di})_{1\leq d\leq D, \, 1\leq i\leq N}\in\R^{D\times N}$ the spin configuration consisting of $N$ spins $\sigma_{\bullet i} \in \R^D$ understood as column vectors. We view $\sigma$ as a $D\times N$ matrix. We sample each spin independently from $P_1$, and thus the distribution of $\sigma$ is given by
\begin{e*}
     \d P_1^{\otimes N}(\sigma)= \otimes_{i=1}^N \d P_1( \sigma_{\bullet i}).
\end{e*}
We let $\xi:\R^{D\times D}\to \R$ be a smooth function that can be expressed as a power series. For each $N$, we consider a centered Gaussian field $(H_N(\sigma))_{\sigma\in\R^{D\times N}}$ with covariance structure given by
\begin{e} \label{e.cov}
    \E \Ll[H_N(\sigma)H_N(\sigma')\Rr] = N \xi \Ll( \tfrac{\sigma\sigma'^\intercal}{N}\Rr),\qquad\forall \sigma,\sigma'\in \R^{D\times N}.
\end{e}
The Hamiltonian $H_N : \R^{D \times N} \to \R$ is the spin glass part of the model. For convenience we will also use the function $\theta:\R^{D\times D}\to\R$ defined by
\begin{align}\label{e.theta=}
    \theta (a) = a\cdot\nabla\xi(a)-\xi(a),\quad\forall a\in\R^{D\times D}.
\end{align}

We now describe the Mattis interaction part of the model. Let $L\in\N$ and let $(\chi_i)_{i\in\N}$ be i.i.d.\ $\R^L$-valued random vectors, independent of $H_N(\sigma)$. Let $d\in\N$ and $h:\R^D\times \R^L \to \R^d$ be a bounded\footnote{Throughout it will be enough to assume that $h$ is bounded on the support of the finite measure $P_1 \otimes \text{Law}(\chi_1)$. } measurable function. For each $i\in\N$, we view $h(\sigma_{\bullet i},\chi_i)$ as a generalized spin and define the mean magnetization
\begin{align}\label{e.m_N=}
    m_N = \frac{1}{N}\sum_{i=1}^N h(\sigma_{\bullet i},\chi_i).
\end{align}
Let $G:\R^d\to\R$ be a continuous function, and we add to the system the Curie--Weiss-type interaction $NG\Ll(m_N\Rr)$. 

At this point, we have described the spin glass part and the Mattis interaction part of our energy function. As explained in the introduction, recently a new approach has been put forward to describe the free energy of spin glass models (without Mattis interaction). In this approach, one needs to enrich the model by adding several new parameters. 

To define the enriched version of the model, we need to add external fields parameterized by increasing paths and driven by Poisson--Dirichlet cascades.  Let $\mcl Q_\infty$ be the collection of bounded increasing càdlàg paths $p:[0,1) \to \S^D_+$, where the monotonicity is understood as $p(s)-p(s')\in\S^D_+$ for $s>s'$. Throughout given $p \in \mcl Q_\infty$ we will define 
\begin{e*}
    p(1) = \lim_{s \to 1-} p(s),
\end{e*}
which is well-defined by monotonicity of $p$.

Throughout, let $\fR$ be the Ruelle probability cascade (RPC) with overlap uniformly distributed over $[0,1]$ (see \cite[Theorem~2.17]{pan}). Precisely, $\fR$ is a random probability measure on the unit sphere of a fixed separable Hilbert space (the exact form of the space is not important), with the inner product denoted by $\alpha\wedge\alpha'$. Let $\alpha$ and $\alpha'$ be independent samples from $\fR$. Then, the law of $\alpha\wedge\alpha'$ under $\E \fR^{\otimes 2}$ is the uniform distribution over $[0,1]$, where $\E$ integrates the randomness of $\fR$.
This overlap distribution uniquely determines $\fR$ (see~\cite[Theorem~2.13]{pan}).
Almost surely, the support of $\fR$ is ultrametric in the induced topology. For rigorous definitions and basic properties, we refer to \cite[Chapter 2]{pan} (also see \cite[Chapter 5]{HJbook}).
We also refer to~\cite[Section~4]{chenmourrat2023cavity} for the construction and properties of $\fR$ useful in this work.

For $q\in\mcl Q_\infty$ and almost every realization of $\fR$, let $(w^q(\alpha))_{\alpha\in\supp\fR}$ be the $\R^D$-valued centered Gaussian process with covariance
\begin{align}\label{e.Ew^qw^q=}
    \E\Ll[ w^q (\alpha)w^q(\alpha')^\intercal\Rr]  = q(\alpha\wedge\alpha').
\end{align}
Its existence and properties are given in~\cite[Section~4]{chenmourrat2023cavity}.
Conditioned on $\fR$, let $(w^q_i)_{i\in\N}$ be i.i.d.\ copies of $w^q$ and also independent of other randomness.

For $N\in \N$, $q\in  \mcl Q_\infty$, $t \geq 0$, the full energy function of the system will be chosen to be 
\begin{align} \label{e.H^t,q;s,x_N(sigma,alpha,chi)=}
\begin{split}
    H^G_{N;t,q}(\sigma,\alpha) &= N G(m_N) + \sqrt{2t} H_N(\sigma) - Nt \xi\Ll(\tfrac{\sigma\sigma^\intercal}{N}\Rr) 
    \\
    &+ \sqrt{2}\sum_{i=1}^N w^q_i(\alpha)\cdot \sigma_{\bullet i} - q(1)\cdot\sigma\sigma^\intercal.   
\end{split} 
\end{align}
We are interested in the infinite-volume limit of the associated free energy, defined as
\begin{align} \label{e.F_N(t,q;s,x)=}
    F_N^G(t,q) = - \frac{1}{N} \log \iint \exp\Ll(H^G_{N;t,q}(\sigma,\alpha) \Rr) \d P_1^{\otimes N}(\sigma) \d \mfk R(\alpha).
\end{align}
We also let $\bar F_N^G(t,q)= \E F_N^G(t,q) = \E_{H_N,\,(w^q_i)_i,\,(\chi_i)_i,\,\fR} F_N^G(t,q)$. In the following, if not specified, $\E$ always averages over the displayed randomness.

As is explained in full detail in Section~\ref{ss.usual}, usual models such as \eqref{e.simple} can be obtained as special cases of the Hamiltonian $H_{N,t,q}^G$. Briefly, to do so one can set $t = \beta^2/2$ and $q = 0$ and make a wise choice for $G$ in order to absorb $- \tfrac{N \beta^2}{2}\xi\Ll(\tfrac{\sigma\sigma^\intercal}{N}\Rr)$ in the Mattis interaction.

Here, we have chosen to introduce the correction terms $Nt \xi\Ll(\sigma\sigma^\intercal/N\Rr)$ and $q(1)\cdot\sigma\sigma^\intercal$ to simplify the forthcoming variational formulas. As explained in Section~\ref{ss.usual}, they are easily removed. These terms are one-half of the variances of $ \sqrt{2t}H_N(\sigma)$ and $\sqrt{2}\sum_{i=1}^Nw^q_i(\alpha)\cdot \sigma_{\bullet i}$, respectively, and therefore play the role of exponential-martingale drifts.

The large deviation principle we prove holds almost surely with respect to the randomness of $(H_N)_{N \geq 1}$, $(\chi_i)_{i \in \N}$, and $\mfk R$. Therefore, throughout we will assume that those random variables are defined on a common probability space $\Omega$ to have a convenient way of stating those results.

Let $x \in \R^d$. When $G(m) = m \cdot x$, at $t = 0$ we have $\bar F^G_N(0,q) =\bar F_1^G(0,q)$ (see~\cite[Proposition~3.2]{chenmourrat2023cavity}). In this case, we set $\psi(q;x)=\bar F_1^{m \mapsto m \cdot x}(0,q)$ which can be expressed as
\begin{align}
    \begin{split}
        \psi(q;x) = -\E\log\iint\exp\Ll(\sqrt{2}w^q(\alpha)\cdot \tau-q(1)\cdot\tau\tau^\intercal+x\cdot h(\tau,\chi_1)\Rr) 
    \\
    \d P_1(\tau)\d\mathfrak{R}(\alpha).
    \end{split} \label{e.psi(q;x)=}
\end{align}

\subsection{Main results}

Throughout, given two metric spaces $\mcl X$ and $\mcl Y$, we use $\|\scdot\|_{\mathrm{Lip}(\mcl X;\mcl Y)}$ to indicate the Lipschitz coefficient of functions from $\mcl X$ to $\mcl Y$. We let $\nabla \xi\lfloor_B$ be the restriction of $\nabla \xi$ to $B$ where  
\begin{align}\label{e.B=}
    B=\Ll\{a\in\R^{D\times D}:\: |a|\leq1\Rr\}.
\end{align}
We let $t_\star \in (0,+\infty)$ be the positive real number defined by 

\begin{align}\label{e.t_star=}
    t^{-1}_\star = \sup_{x'\in\R^d} \Ll\|\nabla_q \psi(\scdot;x')\Rr\|_{\mathrm{Lip}(\mcl Q_\infty;\mcl Q_\infty)}\Ll\|\nabla \xi\lfloor_{B}\Rr\|_{\mathrm{Lip}(B;\R^{D\times D})}.
\end{align}
\begin{remark}[Lower bound for $t_\star$] 
    Below, see \eqref{e.p.psi_smooth_2} in Proposition~\ref{p.psi_smooth}, we justify that for every $x' \in \R^d$, the function $\psi(\scdot;x')$ is differentiable on $\mcl Q_\infty$ and that its gradient $\nabla_q \psi(\scdot;x') : \mcl Q_\infty \to \mcl Q_\infty$ is a Lipschitz function with Lipschitz constant $\leq 16$. This yields the following lower bound 
    \begin{align}\label{e.t_star>}
        t_\star \geq \Ll(16\Ll\|\nabla \xi\lfloor_{B}\Rr\|_{\mathrm{Lip}(B;\R^{D\times D})}\Rr)^{-1} > 0.
    \end{align}
    We note that this bound is very loose and is unlikely to be optimal.
    \qed
\end{remark}
For every $x \in \R^d$, let $f=f(\cdot,\scdot;x) : \R_+ \times \mcl Q_2 \to \R$ denote the unique solution\footnote{Below in Section~\ref{s.short_time} we give a precise meaning to this partial differential equation and explain in what sense it has a unique solution.}  of
\begin{e} \label{e.hj_intro}
    \begin{cases}
        \partial_t f - \int \xi(\nabla f) = 0 \text{ on } (0,+\infty) \times \mcl Q_2 \\
        f(0,\scdot;x) = \psi(\scdot;x).
    \end{cases}
\end{e}
We refer the reader to the discussion at the beginning of Section~\ref{s.short_time} for more details on the precise meaning of \eqref{e.hj_intro} and the associated notion of viscosity solution.

    %
    %
    %
    %
    %
    %
%
\begin{theorem}[Limit free energy] \label{t.free_energy}
    Let $t < t_\star$ and $q \in \mcl Q_\infty$. For every continuous function $G : \R^d \to \R$, we have 
    \begin{e} \label{e.free_energy}
        \lim_{N \to +\infty} \bar F_N^G(t,q) = \inf_{m\in\R^d}\sup_{x\in\R^d}\Ll\{f(t,q;x)+m\cdot x-G(m)\Rr\}.
    \end{e}
Moreover, for every $x\in\R^d$, there is a unique solution $p^x\in\mcl Q_\infty$ of the fixed point equation
\begin{e} \label{e.fixed_point}
        p = \nabla_q \psi(q+t \nabla \xi(p);x) \quad\text{for $p\in \mcl Q_\infty$},
\end{e} 
and we have
\begin{e} \label{e.visco=crit_parisi}
\begin{split}
    f(t,q;x ) = \psi(q +t \nabla \xi(p^x);x) - t \int_0^1 \theta(p^x(u)) \d u \quad\text{and}\quad \nabla_qf(t,q;x) = p^x.
\end{split}
\end{e}
\end{theorem}

The convergence in~\eqref{e.free_energy} also happens almost surely due to the following.

\begin{remark}[Concentration]\label{r.concentration}

    The random variable $F_N^G(t,q)$ concentrates around its expectation $ \E F_N^G(t,q)$ in the large $N$ limit. 
    Indeed, this can be deduced using the standard Gaussian concentration result (e.g.\ \cite[Theorem~1.2]{pan}) for the Gaussian randomness in $F_N^G(t,q)$ together with the Efron--Stein inequality and the boundedness of the function $h(\cdot,\cdot)$ for random variables $(\chi_i)_i$. 
\qed
\end{remark}

The fixed point equation~\eqref{e.fixed_point} can be understood as the critical point condition for the functional on the right-hand side of the first equation in~\eqref{e.visco=crit_parisi}. One can interpret the unique $p^x$ solution of \eqref{e.fixed_point} as the Parisi measure, by identifying $p^x$ with the law of $p^x(U)$ with $U$ the uniform random variable over $[0,1]$.

We also mention that representations similar to \eqref{e.visco=crit_parisi} have been obtained for the mutual information in the sparse stochastic block model (with no Mattis interaction)~\cite[Proposition~1.3]{dominguez2024critical}.

Note that for every $x \in \R^d$, $t < t_\star$ and $q \in \mcl Q_\infty$ we have
\begin{e} \label{e.def.phi}
    f(t,q;x) = \lim_{N \to +\infty} F_N^{m \mapsto x \cdot m}(t,q).
\end{e}
%

\begin{remark}[Free energy of models without Mattis interaction]
    For models without Mattis interaction ($G= 0$), it is not known (but very strongly expected) that the free energy converges in the large $N$ limit. Identifying the limiting value of the free energy is a major problem. In \cite{mourrat2020free,mourrat2020nonconvex} it was shown that for all $t \geq 0$ and $q \in \mcl Q_\infty$,
    \begin{e}
        \liminf_{N \to +\infty} F^0_N(t,q) \geq f(t,q;0),
    \end{e}
    and it was conjectured that this lower bound is tight. As a consequence of \eqref{e.def.phi} we deduce that this conjecture is at least verified for $t < t_\star$. \qed
\end{remark}


\begin{definition}[Large deviation principle]\label{d.ldp}
    Let $(X_N)_{N \geq 1}$ be a sequence of random variables in $\R^d$ with associated probability measures $(\P_N)_{N \geq 1}$ and let $I:\R^d\to [0,\infty]$. We say that the random variables $X_N$ satisfy a large deviation principle under $\P_N$ with rate function $I$ when for every Borel measurable set $A \subset \R^d$, we have  
    \begin{align*}
        -\inf_{ A^\circ} I
        &\leq \liminf_{N\to\infty}\frac{1}{N}\log \P_N\{X_N\in A\}
        \leq \limsup_{N\to\infty}\frac{1}{N}\log \P_N\{X_N\in A\} \\
        &\leq -\inf_{\bar A} I,
    \end{align*}
    where $A^\circ$ and $\bar A$ are the interior and the closure of $A$, respectively.
\end{definition}

The Gibbs measure associated to $H^G_{N;t,q}$ is the (random) probability measure defined on $\R^{D \times N} \times \supp(\mfk R)$ by
\begin{align}\label{e.<>_Ntqsx=}
    \la\cdot\ra_{N;t,q}^G \quad\propto\quad \exp\Ll(H^G_{N;t,q}(\sigma,\alpha)\Rr)\d P_1^{\otimes N}(\sigma)\mathfrak{R}(\alpha).
\end{align}
For brevity, we also denote by $ \la\cdot\ra_{N;t,q}^G$ its tensorization. 

Recall the definition of $f(t,q,\cdot)$ in \eqref{e.def.phi}. We let $ f^*(t,q;\cdot)$ denote the convex dual of $-f(t,q;\cdot)$, namely,
\begin{e}\label{e.f^*=}
    f^*(t,q;m)=  \sup_{x\in\R^d} \{x \cdot m + f(t,q;x)\}, \quad \forall m \in \R^d.
\end{e}
Recall that the random variables $(H^G_{N;t,q})_{N \geq 1}$ and $(\chi_i)_{i \in \N}$ are defined on a common probability space $\Omega$. For a random variable $X : \Omega \to \mfk X$, we let $X^\omega \in \mfk X$ denote the realization of the random variable $X$ on the event $\omega \in \Omega$. Recall the mean magnetization $m_N$ defined in \eqref{e.m_N=}. Our next main result is a (quenched) large deviation principle for the random variables $m_N$ under $\la\cdot\ra^{G}_{N;t,q}$.

\begin{theorem}[Large deviations of $m_N$] \label{t.ldp}
   Let $t < t_\star$ and $q \in \mcl Q_\infty$.  There is a full-measure subset $\Omega' \subset \Omega$ such that the following holds. For every $\omega \in \Omega'$ and every continuous function $G : \R^d \to \R$, the random variable $m_N$ satisfies a large deviation principle under $\langle \cdot \rangle_{N;t,q}^{G,\omega}$ with rate function $I^G_{t,q}$ defined by 
   \begin{e}\label{e.I^G=}
       I^G_{t,q}(m)= -G(m)  + f^*(t,q;m) + \sup_{m'} \left\{ G(m') - f^*(t,q;m') \right\},\, \forall \, m \in \R^d.
   \end{e}
\end{theorem}

For $y\in\S^D$, we define 
\begin{align}\label{e.haty=}
   \hat y(s) : [0,1) \to S^D
\end{align}
to be the constant path with value $y$. For $\kappa\in L^1$, we define
\begin{e}
  \check{\kappa}=(\kappa)^\vee =\int_0^1\kappa(r)\d r \in\S^D.  
\end{e}
When the path $q$ is chosen to be $\hat y$ and particularly $\hat 0$, we show that the representation for the free energy derived in Theorem~\ref{t.free_energy} reduces to a replica symmetric formula (Theorem~\ref{t.rs}). 
This is not surprising and corresponds to the traditional picture of replica symmetry at low $t$ as previously mentioned. 


For every $y\in \S^D_+$ and $x \in \R^d$, we set 
\begin{align}\label{e.phi(y,x)=}
    \phi(y;x)=-\E\log\int \exp\Ll(\sqrt{2y}\eta\cdot\tau-y\cdot\tau\tau^\intercal + x \cdot h(\tau,\chi_1) \Rr)\d P_1(\tau),
\end{align}
where $\eta$ is an independent standard $\R^D$-valued Gaussian vector (i.e.\ $\E\eta\eta^\intercal$ is the $D\times D$-identity matrix) and $\E$ averages over both $\eta$ and $\chi_1$. We denote the derivative in the first variable of $\phi(\scdot;x)$ as $\nabla_y \phi(\scdot;x)$.

\begin{theorem}[Replica symmetric form]\label{t.rs}
Let $x\in\R^d$. Define $g(\scdot,\scdot;x):[0,t_\star)\times \S^D_+\to\R$ by
    \begin{e}
        g(t,y;x)=f(t,\hat y;x).
    \end{e}
    The function $g(\scdot,\scdot;x)$ is differentiable on $[0,t_\star)\times \S^D_+$ and satisfies
    \begin{align}
        \begin{cases}
            \partial_t g(t,y;x) - \xi\big(\nabla_y g(t,y;x)\big)=0 &\quad \text{ for } (t,y)\in(0,t_\star)\times \S^D_+, \\
            g(0,y;x) = \phi(y;x) &\quad \text{ for } y \in \S^D_+. 
        \end{cases}
    \end{align}
    Moreover, at every $(t,y)\in (0,t_\star)\times \S^D_+$, there is a unique solution $z\in\S^D_+$ to the fixed point equation
    \begin{align}\label{e.t.fixed_pt_eqn}
        z' = \nabla_y \phi(y+t\nabla\xi(z');x) \quad\text{for }z'\in\S^D_+,
    \end{align}
    and we have
    \begin{e}\label{e.t.fixed_pt_eqn2}
        g(t,y;x) = \phi(y+t\nabla\xi(z);x) - t \theta(z)\quad\text{and}\quad \nabla_y g(t,y;x) = z.
    \end{e}
\end{theorem}

As before, the fixed point equation~\eqref{e.t.fixed_pt_eqn} can be understood as the critical point condition for the right-hand side of the first equation in~\eqref{e.t.fixed_pt_eqn2}.

\begin{remark}[Gaussian external field]\label{r.Gauss_ext_field}
When $q=\hat y$, we can interpret the model as a non-enriched model with a Gaussian external field.
Recall the covariance structure of $w^q(\alpha)$ from~\eqref{e.Ew^qw^q=}. When $q=\hat y$ for some $y\in\S^D_+$, we have
\begin{align}\label{e.w^haty=sqrtyeta}
    w^{\hat y}(\alpha) \stackrel{\d}{=} \sqrt{y}\,\eta,
\end{align}
where $\eta$ is an independent standard $\R^D$-valued Gaussian vector. Therefore, inside $H^G_{N;t,q}(\sigma,\alpha)$, we have
\begin{align}
    \sqrt{2}\sum_{i=1}^N w^{\hat y}_i(\alpha)\cdot \sigma_{\bullet i} \stackrel{\d}{=} \sqrt{2y}\sum_{i=1}^N \eta_i \cdot \sigma_{\bullet i},
\end{align}
where $(\eta_i)$ consists of i.i.d.\ standard $\R^D$-valued Gaussian vectors. Since this quantity no longer depends on $\alpha$, we can rewrite $H^G_{N;t,\hat y}(\sigma,\alpha)$ simply as $H^G_{N;t,y}(\sigma)$, where $y$ determines the Gaussian external field $\sqrt{2y}\sum_{i=1}^N \eta_i \cdot \sigma_{\bullet i}$. In view of~\eqref{e.psi(q;x)=} and~\eqref{e.phi(y,x)=}, we can use~\eqref{e.w^haty=sqrtyeta} to see that for every $y\in \S^D_+$ and $x \in\R^d$,
\begin{align}\label{e.phi=psi}
    \phi(y;x)=\psi\Ll(\hat y;x\Rr). \qed
\end{align}
\end{remark}

\begin{remark}[Replica symmetric formula for the limit free energy]
At $q=\hat y$ for $y\in\S^D_+$ which includes the non-enriched case $q=\hat 0$, we can thus replace $f(t,q;x)$ appearing in Theorems~\ref{t.free_energy} and~\ref{t.ldp} by the replica-symmetric counterpart $g(t,y;x)$. Recall that we interpret $p^x$ appearing in Theorem~\ref{t.free_energy} as the Parisi measure. Then, in this case, $p^x$ corresponds to $\widehat{z^x}$, where $z^x$ is the unique solution of~\eqref{e.t.fixed_pt_eqn} at $x$. Hence, indeed the system is replica-symmetric. \qed
\end{remark}

When $q$ is not a constant path, then $p^x$ will not be constant. Here, the symmetry breaking is not intrinsic to system but rather caused by $q$.

\begin{remark}[Replica symmetry breaking induced by $q$]\label{r.rsb}
    Fix $x \in \R^d$. Let $(t,q) \in [0,t_\star) \times \mcl Q_2$ and assume that $q$ as a path has a jump at $s\in(0,1)$. The map $\nabla_q \psi(\scdot;x)$ preserves jumps (see Lemma~\ref{l.detect_jump}). Therefore, the unique solution $p^x \in \mcl Q_\infty$ in Theorem~\ref{t.free_energy} has also a jump at $s$. In particular, if $q$ is strictly increasing on $[0,1)$, then so is $p^x$. \qed
\end{remark}



Our main results are stated for models based on vector spin glasses. The same conclusions can be deduced for the following models as well. 

\begin{remark}[Multi-type and multi-species models] \label{r.multi-sp}
Further simplification can be made if $\xi$ only depends on the diagonal of the matrix. We call such a model multi-type. In this setting, only the diagonal in the self-overlap matrix matters, and thus the free energy can be encoded with paths taking values in $\R^D_+$ rather than $S^D_+$. A multi-type model is a special case of multi-species models. In Appendix~\ref{s.multi-sp}, we describe how our main results can be adapted to multi-species models. We note that passing back and forth between vector models and multi-species models is largely a technical matter and does not introduce significant new challenges. Essentially all the results mentioned so far on the free energy of vector spin glasses remain valid in this setting~\cite{barcon,pan.multi}. \qed 
\end{remark}

\subsection{Recovering concrete models from the general setting} \label{ss.usual}

The main results of this paper are stated for very general models as introduced in Section~\ref{ss.setting}. Let us explain how to recover results about more usual models, like \eqref{e.simple}, from those.  

Given $h:\R^D\times \R^L\to\R^d$ in~\eqref{e.F_N(t,q;s,x)=}, we define $\tilde h:\R^D\times \R^L\to \R^d\times \S^D$ through
\begin{e}
    \tilde h(\tau,\chi) = \Ll(h(\tau,\chi),\ \tau\tau^\intercal\Rr).
\end{e}
Henceforth, we isometrically identify $\R^d\times \S^D$ with $\R^{\tilde d}$ for some $\tilde d \in\N$. Given $G:\R^d\to\R$ in~\eqref{e.F_N(t,q;s,x)=}, we introduce $\tilde G:\R^d\times \S^D\to\R$ given by
\begin{e}
    \tilde G(m,r) =G(m) + \frac{\beta^2}{2}\xi(r).
\end{e}
Then, we consider $H_{N;t,q}^G$ in~\eqref{e.H^t,q;s,x_N(sigma,alpha,chi)=} with $G,h,d$, therein substituted with $\tilde G,\tilde h,\tilde d$ introduced above and with $t=\beta^2/2$ and $q=\hat 0$ (constant path taking the value $0$), we denote the resulting Hamiltonian $\tilde H_N$. Observe that we have  
\begin{align}
   \tilde H_N(\sigma) &= \beta H_N(\sigma) + N \tilde G \left( \frac{1}{N} \sum_{i = 1}^N \tilde h(\sigma_i, \chi_i) \right) - \frac{N \beta^2}{2} \xi \left( \frac{\sigma \sigma^\intercal}{N}\right) \\
    &= \beta H_N(\sigma) + N G \left( \frac{1}{N} \sum_{i  = 1}^N h(\sigma_{\bullet i}, \chi_i) \right).
\end{align}
In particular, observe that the model described in \eqref{e.simple} is recovered by choosing $D=2$, $d=2$, $L = 2$, $h(\sigma,\chi) = (\text{sgn}(\sigma_1 \chi_1),\text{sgn}(\sigma_2 \chi_2))$, $G(m) = m_1+m_2$, $\xi(a) = a_{11}a_{22}$, and $P_1$ is the uniform measure on $\{-1,1\}^2$.

\begin{remark}[About Theorem~\ref{t.parisi+ldp.basic}] 
     Observe that thanks to the specialization of the general models from Section~\ref{ss.setting} to \eqref{e.simple} given in this section, Theorem~\ref{t.parisi+ldp.basic} can (almost) be recovered from Theorem~\ref{t.free_energy}, Theorem~\ref{t.ldp} and Theorem~\ref{t.rs}. The small difference being that in those results the paths in $\mcl Q$ are $S^2_+$ valued, while in Theorem~\ref{t.parisi+ldp.basic} the expression for the limit free energy comes from constant $\R^2_+$ valued paths. This difference originates from the fact that the model described in \eqref{e.simple} is a multi-type model in the sense defined in Remark~\ref{r.multi-sp}. Indeed, $\xi(a) = a_{11}a_{22}$ does not depend on the whole input matrix but only on its diagonal elements. To recover Theorem~\ref{t.parisi+ldp.basic} as stated, one can use the version of the main results stated in Appendix~\ref{s.multi-sp} instead. The lower bound $\beta_\star \geq \frac{1}{2\sqrt{2}}$ comes from setting $\beta_\star = \sqrt{2 t_\star}$ and the lower bound \eqref{e.t_star>} for $t_\star$.\qed
\end{remark}

\subsection{Organization of the paper} 

In Section~\ref{s.psi_reg}, we establish some regularities of the function $\psi$ that are useful later on; these results are gathered in Proposition~\ref{p.psi_smooth}. 
In Section~\ref{s.short_time}, we show that the solution of \eqref{e.hj_intro} is differentiable on the time interval $[0,t_\star)$. Finally, in Section~\ref{s.proof_main_results}, we prove our main results. We start with a proof of \eqref{e.def.phi}, this is the key aspect of our argument, and it relies on the content of all the previous sections, see Corollary~\ref{c.limF_N=f}. We also crucially verify, thanks to \eqref{e.def.phi}, that $x \mapsto \lim_{N \to +\infty} F_N^{x \mapsto x \cdot m}(t,q)$ is continuously differentiable in $x$, see Lemma~\ref{l.varphi}. Once those two key results are proven, we can deduce Theorem~\ref{t.free_energy} and Theorem~\ref{t.ldp} using classical results from large deviation theory.

\subsection{Acknowledgements}

We would like to warmly thank Jean-Christophe Mourrat for many useful inputs and interesting discussions during the conception and writing of this paper. HBC acknowledges funding from the NYU Shanghai Start-Up Fund and support from the NYU–ECNU Institute of Mathematical Sciences at NYU Shanghai.

\section{Regularity of \texorpdfstring{$\psi$}{psi}} \label{s.psi_reg}

\begin{definition}[Fréchet differentiability]
    Let $(E,\,|\,\cdot\,|_E)$ be a Banach space and denote by $\langle \,\cdot\,,\, \cdot\, \rangle_E$ the canonical pairing between $E$ and $E^*$. Let $G$ be a subset of $E$ and let $q \in G$. A function $g : G \to \R$ is \textbf{Fréchet differentiable} at $q \in G$ if there exists a unique $y^* \in E^*$ such that the following holds:
    \begin{e*}
        \lim_{\substack{q' \to q \\ q' \in G}} \frac{g(q') - g(q) - \langle y^*, q'-q \rangle_E}{|q'-q|_E} = 0.
    \end{e*}
    In such a case, we call $y^*$ the \textbf{Fréchet derivative} of $g$ at $q$, and we denote it by $\nabla g(q)$ or $\nabla_q g(q)$.

    More generally, let $(E',|\scdot|_{E'})$ be another Banach space. Let $G, G'$ be subsets of $E$ and $E'$, respectively, and let $q\in G$. A function $g:G\to G'$ is \textbf{Fréchet differentiable} at $q\in G$ if there exists a unique bounded operator $A:E\to E'$ such that
    \begin{e*}
        \lim_{\substack{q' \to q \\ q' \in G}} \frac{\Ll|g(q') - g(q) - A(q'-q)\Rr|_{E'}}{|q'-q|_E} = 0.
    \end{e*}
    Again, in such a case, we denote $A$ as $\nabla g(q)$ or $\nabla_q g(q)$.
\end{definition}

We will use the operator norm on the set of bounded operators $A : E \to E'$, defined as
\begin{e} \label{e.def.opnorm}
    \| A \|_{\mathrm{Op}(E,E')} = \sup_{q \in E \setminus \{0\}} \frac{|Aq|_{E'}}{|q|_E}.
\end{e}

\begin{proposition}[Regularity of $\psi$]\label{p.psi_smooth}
The following holds, where we write $L^\sfp=L^\sfp([0,1);\S^D)$ for $\sfp\in[1,\infty]$.
\begin{enumerate}
    \item \label{i.p.psi_smooth(1)} The function $\psi$ can be extended to $\mcl Q_1\times \R^d$ and satisfies, for any $q,q'\in\mcl Q_1$ and $x,x'\in\R^d$,
\begin{align}\label{e.psi_Lip}
    \Ll|\psi(q;x) - \psi(q';x')\Rr|\leq |q-q'|_{L^1} + |h|_{L^\infty} |x-x'|.
\end{align}

\item \label{i.p.psi_smooth(2)} For each $x\in\R^d$, the restriction of the function $\psi(\scdot;x)$ to $\mcl Q_2$ is Fréchet differentiable everywhere. The derivative $\nabla_q\psi(\scdot;x)$ can be extended to $\mcl Q_1$ and satisfies, for any $q,q'\in\mcl Q_1$, $x,x'\in\R^d$, and $\sfr\in[1,+\infty]$,
\begin{gather}
    \nabla_q \psi(q;x)\in\mcl Q,\quad \Ll|\nabla_q\psi(q;x)\Rr|_{L^\infty}\leq 1,\label{e.p.psi_smooth_1}
    \\
    \Ll|\psi(q';x)-\psi(q;x)-\la q'-q,\,\nabla_q\psi(q;x)\ra_{L^2}\Rr|\leq 8|q-q'|_{L^2}^2,\label{e.frechet_bound}
    \\
    \Ll|\nabla_q \psi(q;x) - \nabla_q \psi(q';x')\Rr|_{L^\sfr}\leq 16|q-q'|_{L^\sfr}+ 4|h|_{L^\infty}|x-x'|.\label{e.p.psi_smooth_2}
\end{gather}

\item \label{i.p.psi_smooth(3)} For each $q\in\mcl Q_1$, the function $\psi(q;\scdot)$ is differentiable everywhere on $\R^d$ and its derivative satisfies, for $q,q'\in\mcl Q_1$, $x,x'\in\R^d$,
\begin{gather}
    |\nabla_x\psi(q;x)|\leq |h|_{L^\infty} ,\label{e.smooth(3).1}
    \\
    \Ll|\psi(q;x')-\psi(q;x)-(x'-x)\cdot\nabla_x\psi(q;x)\Rr|\leq |h|_{L^\infty}^2|x-x'|^2,\label{e.smooth(3).1.5}
    \\
    \Ll|\nabla_x \psi(q;x) - \nabla_x \psi(q';x')\Rr|\leq 8|h|_{L^\infty}|q-q'|_{L^1}+ 2|h|^2_{L^\infty}|x-x'|.\label{e.smooth(3).2}
\end{gather}

\item  \label{i.p.psi_smooth(4)}
For each $x\in\R^d$, the function $\nabla_q\psi(\scdot;x) :\: \mcl Q_{2\sfs} \to L^\sfs$ is Fréchet differentiable everywhere for every $\sfs \in[1,+\infty]$ and satisfies, for $q,q'\in\mcl Q_{2\sfs}$,
\begin{align}\label{e.psi_smooth(4).0}
    \Ll| \nabla_q\psi(q';x)-\nabla_q\psi(q;x) - \nabla^2_q\psi(q;x)(q'-q)\Rr|_{L^\sfs}\leq C|q-q'|_{L^{2\sfs}}^2/2,
\end{align}
for some absolute constant $C>0$ independent of $x$ and $\sfs$, where $\nabla_q^2\psi(q;x)$ is viewed as a bounded linear operator from $L^{2\sfs}$ to $L^\sfs$ (but actually enjoys better regularity as in~\eqref{e.psi_smooth(4).1}).
The derivative $\nabla^2_q\psi(\scdot;x)$ can be extended to $\mcl Q_1$ and satisfies, for every $q,q'\in\mcl Q_1$ and every $\sfp,\sfq\in[1,+\infty]$ satisfying $\frac{1}{\sfp}+\frac{1}{\sfq}\leq 1$,
\begin{gather}
    \Ll\|\nabla^2_q\psi(q;x)\Rr\|_{\mathrm{Op}(L^\sfp;L^\sfp)} \leq 16, \label{e.psi_smooth(4).1}
    \\
    \Ll\|\nabla^2_q\psi(q;x)-\nabla^2_q\psi(q';x)\Rr\|_{\mathrm{Op}(L^\sfp; L^\frac{\sfp\sfq}{\sfp+\sfq})} \leq C\Ll|q-q'\Rr|_{L^\sfq},\label{e.psi_smooth(4).2}
\end{gather}
for the same constant $C$ in~\eqref{e.psi_smooth(4).0}.
\end{enumerate}
\end{proposition}

One useful case of~\eqref{e.psi_smooth(4).2} is when $\sfp=\sfq=2\sfs$ for some $\sfs\in[1,+\infty]$ and thus $\frac{\sfp\sfq}{\sfp+\sfq}=\sfs$.

\begin{proof}[Proof of Proposition~\ref{p.psi_smooth}]
The argument is based on computing directional derivatives in $q$ and $x$. The computation is straightforward for the latter, and we explain the former. Fix any $q,q'\in\mcl Q_\sfs$ and let $\lambda\in[0,1]$. We want to compute the derivative in $\lambda$ along $\lambda\mapsto q+\lambda(q'-q)$. The key observation is that $w^{q+\lambda(q'-q)}(\alpha)$ inside $\psi(q+\lambda(q'-q))$ (see~\eqref{e.psi(q;x)=}) has the same law as $\sqrt{1-\lambda}w^q(\alpha) + \sqrt{\lambda}w^{q'}(\alpha)$ where $w^q(\alpha)$ and $w^{q'}(\alpha)$ are independent. Hence, we can carry out the computation in terms of $\sqrt{1-\lambda}w^q(\alpha) + \sqrt{\lambda}w^{q'}(\alpha)$. The recipe is as follows: first, we differentiate a quantity in $\lambda$ to get an expression involving $w^q(\alpha)$ and $w^{q'}(\alpha)$, and then we apply the Gaussian integration by parts~\cite[Theorem~4.6]{HJbook} to them to get the final expression.
Since these procedures are standard, we omit the intermediate steps and refer to the computations in the proof of~\cite[Proposition~5.1]{chenmourrat2023cavity} as an example. 

Throughout the proof, we use the following notation. For $\kappa\in L^\infty([0,1];\S^D)$ and $\ell,\ell'\in\N$, we write
\begin{align*}
    R^{\ell,\ell'}_\kappa = \kappa\Ll(\alpha^\ell\wedge\alpha^{\ell'}\Rr)\cdot \tau^\ell(\tau^{\ell'})^\intercal.
\end{align*}
where $(\tau^\ell)_{\ell\in\N}$ are independent samples from some Gibbs measure that will be clear from the context.
For each $a\in \R^d$ and $\ell\in\N$, we write
\begin{align*}
    h^\ell = h\Ll(\tau^\ell,\chi_1\Rr) \qquad \text{and}\qquad h^\ell_a = a\cdot h\Ll(\tau^\ell,\chi_1\Rr).
\end{align*}
In the following, let $\lambda\in[0,1]$. Recall the definition of the Gibbs measure $\langle \cdot \rangle_{N;t,q}^G$ in \eqref{e.<>_Ntqsx=}. For $q,q'\in \mcl Q_\infty$ and $x,x'\in\R^d$, we write
\begin{align*}
    \la\scdot\ra_\lambda =\la\scdot\ra_{1;0,q+\lambda(q'-q)}^{m \mapsto m \cdot(x+\lambda(x'-x))},
\end{align*}
where it is also possible that we take $q=q'$ or $x=x'$, which will be clear from the context. 

Since $\alpha^\ell\wedge\alpha^{\ell'}$ distributes uniformly over $[0,1]$ under $\E\la\scdot\ra_\lambda$ for $\ell\neq\ell'$ and $\Ll|\tau\tau^\intercal\Rr|\leq1$ due to the support of $\tau$ on the unit ball in $\R^D$, we often use the the following estimates:
\begin{align}\label{e.HolderE<>}
    \E\la\Ll| R_{\kappa}^{\ell,\ell'}\Rr|\ra_\lambda \leq |\kappa|_{L^1}\qquad \text{and}\qquad  \E\la\Ll| R_{\kappa}^{\ell,\ell'}R_{\eta}^{\ell'',\ell'''}\Rr|\ra_\lambda \leq |\kappa|_{L^\sfp}|\eta|_{L^\sfq},
\end{align}
where $\ell\neq \ell'$, $\ell''\neq\ell'''$, $\kappa,\eta\in L^\infty([0,1];\S^D)$, and $\frac{1}{\sfp}+\frac{1}{\sfq}=1$. The last inequality follows from the H\"older's inequality applied to the measure $\E\la\scdot\ra_\lambda$. Later, we will also use the generalization of the last one to three components.

We also mention that the correction term $-q(1)\cdot\tau\tau^\intercal$ in~\eqref{e.psi(q;x)=} becomes $-(q+\lambda(q'-q))(1)\cdot\tau\tau^\intercal$ when we compute the direction derivative. The presence of this term will cancel all $R^{\ell,\ell}_\kappa$ and thus in the final expression we only have terms $R^{\ell,\ell'}_\kappa$ with $\ell\neq\ell'$. This is an important reason for the usefulness of the correction.

Lastly, slightly abusing the notation, we denote by $\la\scdot,\,\scdot\ra_{L^2}$ not only the inner product in $L^2([0,1];\S^D)$ but also the natural paring in $L^\sfp([0,1];\S^D)$ and its dual.
We also use the shorthand $L^\sfp=L^\sfp([0,1];\S^D)$. The only exception is that $L^\infty$ in $|h|_{L^\infty}$ means $L^\infty(\R^D\times \R^L;\R^d)$.

With the above setup, we are ready to prove the lemma part by part.

\textit{Part~\eqref{i.p.psi_smooth(1)}.}
Recall that $\psi$ is initially defined on $\mcl Q_\infty\times \R^d$ as in~\eqref{e.psi(q;x)=}.
For $q,q'\in \mcl Q_\infty$ and $x,x'\in\R^d$, we can compute
\begin{align}\label{e.d/dlambdapsi()=E<R>-E<h>}
    \frac{\d}{\d\lambda}\psi\Ll(q+\lambda(q'-q);x+\lambda(x'-x)\Rr) = \E \la R^{1,2}_{q'-q}\ra_\lambda - \E \la h^1_{x'-x} \ra_\lambda.
\end{align}
Using~\eqref{e.HolderE<>}, the last term is bounded by $|q'-q|_{L^1}+ |h|_{L^\infty}$. Therefore, we get that~\eqref{e.psi_Lip} holds for $q,q'\in \mcl Q_\infty$ and $x,x'\in\R^d$. By this continuity, we can extend the definition of $\psi$ to $\mcl Q_1\times \R^d$ and accordingly extend~\eqref{e.psi_Lip} as desired.

\textit{Part~\eqref{i.p.psi_smooth(2)}.}
For each $q\in\mcl Q_\infty$ and $x\in\R^d$, we define $\nabla_q\psi(q;x)\in L^2$ as the element determined through
\begin{align}\label{e.nabla_qpsi(q;x)=}
    \la \kappa,\, \nabla_q\psi(q;x)\ra_{L^2} = \E \la R^{1,2}_\kappa\ra_0,\quad\forall \kappa \in L^\infty.
\end{align}
Indeed, setting $x=x'$ in~\eqref{e.d/dlambdapsi()=E<R>-E<h>}, we get
\begin{align}\label{e.d/dlambdapsi=<>}
    \frac{\d}{\d\lambda}\psi(q+\lambda(q'-q);x)= \la q'-q, \nabla_q\psi(q;x)\ra_{L^2},\qquad\forall q,q'\in \mcl Q_\infty,\, x\in\R^d.
\end{align}
This suggests that $\nabla_q\psi(q;x)$ is the correct derivative, which we will justify later after gathering enough properties.

Since $\la R^{1,2}_{\kappa}\ra_\lambda \leq |\kappa|_{L^1}$ due to~\eqref{e.HolderE<>} and the density of $L^\infty$ in $L^2$, we get from~\eqref{e.nabla_qpsi(q;x)=} that
\begin{align}\label{e.|nablapsi|<1}
    \Ll|\nabla_q\psi(q;x)\Rr|_{L^\infty} \leq 1,\qquad\forall q \in\mcl Q_\infty,\, x\in\R^d.
\end{align}
For $\kappa\in L^\infty$, $q,q'\in\mcl Q_\infty$, and $x\in\R^d$, we can compute
\begin{align}\label{e.d/dlambda<k,nablapsiq_lambda>=}
    \frac{\d}{\d\lambda}\la \kappa,\, \nabla_q\psi(q+\lambda(q'-q);x)\ra_{L^2}=2\E\la R^{1,2}_\kappa\Ll(R^{1,2}_{q'-q}-4R^{1,3}_{q'-q}+3R^{3,4}_{q'-q}\Rr)\ra_{\lambda}.
\end{align}
By~\eqref{e.HolderE<>}, the right-hand side is bounded by $16|\kappa|_{L^{\sfr^*}}|q-q'|_{L^\sfr}$ with $\frac{1}{\sfr^*}+\frac{1}{\sfr}=1$. 
Integrating in $\lambda$ and using the density of $L^\infty$ in $L^{\sfr^*}$, we get
\begin{align}\label{e.|nablapsi-nablapsi|<16|q-q'|}
    \Ll|\nabla_q\psi(q;x) - \nabla_q\psi(q';x)\Rr|_{L^\sfr}\leq 16|q-q'|_{L^\sfr},\qquad\forall q,q'\in\mcl Q_\infty,\, x\in\R^d.
\end{align}
Setting $\sfr=1$ and using~\eqref{e.|nablapsi|<1}, we extend $\nabla\psi(q,x)$ by continuity to $q\in\mcl Q_1$ and achieve the $L^\infty$-bound in~\eqref{e.p.psi_smooth_1}. First property in~\eqref{e.p.psi_smooth_1} is proved in~\cite[Corollary~5.2]{chenmourrat2023cavity}.

Setting $\kappa=q'-q$ in~\eqref{e.d/dlambda<k,nablapsiq_lambda>=} and $\sfr=2$ in the bound mentioned below~\eqref{e.d/dlambda<k,nablapsiq_lambda>=} and comparing with~\eqref{e.d/dlambdapsi=<>}, we get $\Ll|\frac{\d^2}{\d\lambda^2}\psi(q+\lambda(q'-q);x)\Rr|\leq 16|q-q'|_{L^2}^2$. This, along with~\eqref{e.d/dlambdapsi=<>} and Taylor's expansion, gives that for any fixed $x\in\R^d$,
\begin{align}\label{e.|...|<8|q-q'|^2}
    \Ll|\psi(q';x)-\psi(q;x)-\la q'-q,\,\nabla_q\psi(q;x)\ra_{L^2}\Rr|\leq 8|q-q'|^2_{L^2},\quad \forall q,q'\in\mcl Q_\infty.
\end{align}
Since we have shown that $\psi(\scdot;x)$ and $\nabla_q\psi(\scdot;x)$ can be extended continuously to $\mcl Q_1$ and that the latter is also bounded, we can now extend~\eqref{e.|...|<8|q-q'|^2} to $q,q'\in\mcl Q_2$. This proves that $\psi(\scdot;x)$ is Fréchet differentiable everywhere on $\mcl Q_2$ satisfying~\eqref{e.frechet_bound} and that $\nabla_q\psi(q;x)$ defined via~\eqref{e.nabla_qpsi(q;x)=} is its derivative.

To complete this part, it remains to show~\eqref{e.p.psi_smooth_2}. 
For $\kappa\in L^\infty$, $q\in\mcl Q_\infty$ and $x,x'\in\R^d$, we compute from~\eqref{e.nabla_qpsi(q;x)=} that
\begin{align*}
    \frac{\d}{\d\lambda}\la \kappa,\, \nabla_q \psi(q; x+\lambda(x'-x))\ra_{L^2}= \E \la R^{1,2}_\kappa\Ll(h^1_{x'-x}+ h^2_{x'-x}- 2h^3_{x'-x}\Rr) \ra_\lambda.
\end{align*}
Using~\eqref{e.HolderE<>} to bound the right-hand side by $4|h|_{L^\infty}|\kappa|_{L^1}|x-x'|$ and using the density of $L^\infty$ in $L^1$, we get
\begin{align*}
    \Ll|\nabla_q \psi(q;x)-\nabla_q \psi(q;x')\Rr|_{L^\infty}\leq 4|h|_{L^\infty}|x-x'|,\qquad\forall q\in\mcl Q_\infty,\  x,x'\in\R^d.
\end{align*}
As before, we can extend this relation to $q\in\mcl Q_1$, which along with~\eqref{e.|nablapsi-nablapsi|<16|q-q'|} extended to $\mcl Q_1$ yields~\eqref{e.p.psi_smooth_2}.

\textit{Part~\eqref{i.p.psi_smooth(3)}.}
We can easily compute
\begin{align*}
    \nabla_x\psi(q;x) = - \E \la h^1\ra_0,\quad \Ll|\nabla_x\psi(q;x)\Rr|\leq |h|_{L^\infty},\qquad\forall q\in\mcl Q_\infty,\, x\in\R^d.
\end{align*}
For $q,q'\in \mcl Q_\infty$ and $x\in\R^d$, we can then compute
\begin{align*}
    \frac{\d}{\d\lambda}\nabla_x \psi(q+\lambda(q'-q);x) = 4 \E \la h^1\Ll(R^{1,2}_{q'-q}-R^{2,3}_{q'-q}\Rr) \ra_\lambda.
\end{align*}
Using~\eqref{e.HolderE<>} to bound the right-hand side, we get
\begin{align}\label{e.|nabla_xpsi-nabla_xpsi|<8|h||q-q'|}
    \Ll|\nabla_x \psi(q;x)-\nabla_x \psi(q';x)\Rr|\leq 8|h|_{L^\infty}|q-q'|_{L^1} \qquad \forall q\in\mcl Q_\infty,\, x\in\R^d.
\end{align}
Therefore, we can extend $\nabla_x\psi(q;x)$ by continuity to $q\in\mcl Q_1$ and get the bound in~\eqref{e.smooth(3).1}. Notice that we have yet to verify that the extended $\nabla_x\psi(q;x)$ is indeed the derivative at $q\in\mcl Q_1$.

To do so, we start by computing, for $q\in\mcl Q_\infty$ and $x,x'\in\R^d$,
\begin{align}\label{e.d/dlambdanabla_xpsi=-E}
    \frac{\d}{\d\lambda}\nabla_x\psi(q;x+\lambda(x'-x)) = - \E \la h^1\Ll(h^1_{x'-x}-h^2_{x'-x}\Rr) \ra_\lambda,
\end{align}
which implies
\begin{align*}
    \Ll|\nabla_x\psi(q;x)-\nabla_x\psi(q;x')\Rr| \leq 2|h|_{L^\infty}^2|x-x'|,\quad\forall q\in\mcl Q_\infty,\  x,x'\in\R^d.
\end{align*}
Again, \eqref{e.|nabla_xpsi-nabla_xpsi|<8|h||q-q'|} allows us to extend this bound to $q\in \mcl Q_1$, which along with the extension of~\eqref{e.|nabla_xpsi-nabla_xpsi|<8|h||q-q'|} to $\mcl Q_1$ yields~\eqref{e.smooth(3).2}.

Taking the scalar product of both sides in~\eqref{e.d/dlambdanabla_xpsi=-E} with $x'-x$, we get $\Ll|\frac{\d^2}{\d\lambda^2}\psi(q;x+\lambda(x'-x))\Rr|\leq 2|h|_{L^\infty}^2|x-x'|^2$, which by Taylor's expansion gives that, uniformly in $x,x'\in\R^d$,
\begin{align*}
    \Ll|\psi(q;x')-\psi(q;x)-(x'-x)\cdot\nabla_x\psi(q;x)\Rr|\leq |h|_{L^\infty}^2|x-x'|^2,\quad \forall q \in\mcl Q_\infty.
\end{align*}
As before, we can extend this by continuity to $q\in\mcl Q_1$. This verifies that $\nabla\psi(q;\scdot)$ is differentiable at $q\in\mcl Q_1$, \eqref{e.smooth(3).1.5}, and $\nabla_x\psi(q;x)$ is indeed the derivative.

\textit{Part~\eqref{i.p.psi_smooth(4)}.}
We fix any $x\in\R^d$ and write $\psi=\psi(\scdot;x)$.
For any $\sfp\in[0,+\infty]$ and $q\in\mcl Q_\infty$, we define $\nabla^2_q\psi(q)$ as a linear operator from $L^\sfp$ to $L^\sfp$ by setting
\begin{align}\label{e.nabla^2psi=}
    \la \kappa,\, \nabla^2_q \psi(q)\eta\ra_{L^2} = 2\E\la R^{1,2}_\kappa\Ll(R^{1,2}_\eta-4R^{1,3}_\eta+3R^{3,4}_\eta\Rr)\ra_0,\quad\forall\kappa,\eta\in L^\infty.
\end{align}
Notice that this definition can be made simultaneously for all $\sfp$. Using~\eqref{e.HolderE<>}, we can bound the right-hand side by $16|\kappa|_{L^{\sfp^*}}|\eta|_{L^\sfp}$ with $\frac{1}{\sfp^*}+\frac{1}{\sfp}=1$, which implies
\begin{align}\label{e.|nabla^2psi|<16}
    \Ll\|\nabla^2_q\psi(q)\Rr\|_{\mathrm{Op}(L^\sfp;L^\sfp)} \leq 16,\qquad\forall q \in \mcl Q_\infty.
\end{align}
Comparing with~\eqref{e.d/dlambda<k,nablapsiq_lambda>=}, we get, for $q,q'\in\mcl Q_\infty$ and $\kappa \in L^\infty$,
\begin{align}\label{e.part(4)_1}
    \frac{\d}{\d\lambda}\la \kappa,\, \nabla_q\psi(q+\lambda(q'-q))\ra_{L^2}=  \la \kappa,\, \nabla^2_q \psi(q+\lambda(q'-q))(q'-q)\ra_{L^2}.
\end{align}
We will extend the definition of $\nabla^2_q\psi$ to $\mcl Q_1$ and verify that it is indeed a Fréchet derivative. But first, we need to derive some estimates.
For $\kappa,\eta\in L^\infty$ and $q,q'\in\mcl Q_\infty$, we can compute from~\eqref{e.nabla^2psi=} that
\begin{align}\label{e.part(4).2}
    &\frac{\d}{\d\lambda}\la \kappa,\, \nabla^2_q\psi(q+\lambda(q'-q))\eta\ra_{L^2} \\
    &\qquad= \E\la R^{1,2}_\kappa\Ll(R^{1,2}_\eta-4R^{1,3}_\eta+3R^{3,4}_\eta\Rr)\mcl P\Ll(R_{q'-q}\Rr)\ra_\lambda,
\end{align}
where $\mcl P$ is a homogeneous polynomial of degree one and depends on finitely many entries in $R_{q'-q} = \big(R^{\ell,\ell'}_{q'-q}\big)_{\ell\neq \ell'}$. Using an extension of the estimate in~\eqref{e.HolderE<>}, we can bound the right-hand side in~\eqref{e.part(4).2} by $C|\kappa|_{L^\sfr}|\eta|_{L^\sfp}|q-q'|_{L^\sfq}$ for some absolute constant $C>0$, where $\frac{1}{\sfr}+\frac{1}{\sfp}+\frac{1}{\sfq}=1$. Notice that the conjugate of $\sfr$ is $\sfr^*= \frac{\sfp\sfq}{\sfp+\sfq}$. From this, we can get
\begin{align}\label{e.|nabla^2psi-nabla^2psi|_Op<|q-q'|}
    \Ll\|\nabla^2_q\psi(q)-\nabla^2_q\psi(q')\Rr\|_{\mathrm{Op}(L^\sfp; L^\frac{\sfp\sfq}{\sfp+\sfq})}\leq C\Ll|q-q'\Rr|_{L^\sfq},\quad\forall q,q'\in\mcl Q_\infty.
\end{align}
In this relation, setting $\sfr=\sfp=\infty$ and $\sfq=1$, we can thus extend by continuity the definition of $\nabla^2_q\psi(q)$ to every $q\in\mcl Q_1$ as a bounded linear operator from $L^\infty$ to $L^1$.
Then, we want to accordingly extend~\eqref{e.|nabla^2psi|<16}. Let $q\in \mcl Q_1$ and let $(q_n)_{n\in\N}$ be a sequence in $\mcl Q_\infty$ convergent to $q$ in $L^1$. Then, for any $\kappa,\eta\in L^\infty$, we have
\begin{align*}
    \la \kappa,\, \nabla^2_q\psi(q)\eta\ra_{L^2} \leq \la \kappa,\, \nabla^2_q\psi(q_n)\eta\ra_{L^2} + \la \kappa,\, \Ll(\nabla^2_q\psi(q)-\nabla^2_q\psi(q_n)\Rr)\eta\ra_{L^2}.
\end{align*}
By~\eqref{e.|nabla^2psi-nabla^2psi|_Op<|q-q'|} with $\sfr=\sfp=\infty$ and $\sfq=1$, the last term vanishes as $n\to\infty$. Taking this limit and using~\eqref{e.|nabla^2psi|<16} for $q_n$, we can obtain from the above display that~\eqref{e.|nabla^2psi|<16} holds for $q\in\mcl Q_1$, which gives~\eqref{e.psi_smooth(4).1} as desired.
Therefore, for every $q\in\mcl Q_1$, $\nabla^2_q\psi(q)$ is also a bounded operator from $L^\sfp\to L^\sfp$. Due to $\sfp\geq \frac{\sfp\sfq}{\sfp+\sfq}$, it is also an operator from $L^\sfp\to L^\frac{\sfp\sfq}{\sfp+\sfq}$. Therefore, we can extend~\eqref{e.|nabla^2psi-nabla^2psi|_Op<|q-q'|} by continuity to $q,q'\in \mcl Q_1$, which is exactly~\eqref{e.psi_smooth(4).2}.

Lastly, we verify that $\nabla^2_q\psi(q)$ indeed defines a Fréchet derivative.
Setting $\eta=q'-q$ in~\eqref{e.part(4).2}, comparing it with~\eqref{e.part(4)_1}, and using the estimate mentioned below~\eqref{e.part(4).2} with $\sfr=\frac{\sfs}{\sfs-1}$ and $\sfp=\sfq=2\sfs$, we get, for $\kappa\in L^\infty$ and $q,q'\in\mcl Q_\infty$,
\begin{align*}
    \Ll|\frac{\d^2}{\d\lambda^2}\la \kappa,\, \nabla_q\psi(q+\lambda(q'-q))\ra_{L^2}\Rr|\leq C|\kappa|_{L^\frac{\sfs}{\sfs-1}} |q-q'|_{L^{2\sfs}}^2.
\end{align*}
Using this, Taylor's expansion, and~\eqref{e.part(4)_1}, we get
\begin{align*}
    \Ll|\la\kappa,\, \nabla_q\psi(q')-\nabla_q\psi(q) - \nabla^2_q\psi(q)(q'-q)\ra_{L^2}\Rr|\leq C|\kappa|_{L^\frac{\sfs}{\sfs-1}} |q-q'|_{L^{2\sfs}}^2/2.
\end{align*}
and thus
\begin{align*}
    \Ll| \nabla_q\psi(q')-\nabla_q\psi(q) - \nabla^2_q\psi(q)(q'-q)\Rr|_{L^\sfs}\leq C|q-q'|_{L^{2\sfs}}^2/2,\quad\forall q,q'\in\mcl Q_\infty.
\end{align*}
We can extend this to $q,q'\in\mcl Q_{2\sfs}$ by using~\eqref{e.p.psi_smooth_2} (with $\sfr =\sfs$), \eqref{e.psi_smooth(4).1} (with $\sfp=\sfs$), and~\eqref{e.psi_smooth(4).2} (with $\sfp=\sfq=2\sfs$). This implies the statement on the Fréchet differentiability of this part and~\eqref{e.psi_smooth(4).0}, completing the proof.
\end{proof}

We also record a simple result based on Proposition~\ref{p.psi_smooth}.
Define the set
\begin{align}\label{e.Q_infty,<1}
    \mcl Q_{\infty, \leq 1} = \Ll\{p\in\mcl Q_\infty:\: |p|_{L^\infty}\leq 1\Rr\}.
\end{align}
For every $t\geq0$, $p ,q \in \mcl Q_\infty$, and $x\in\R^d$, we define
    \begin{e} \label{e.P_t,q(p;x)=}
        \msc P_{t,q}(p;x) = \psi(q +t \nabla \xi(p);x) - t \int_0^1 \theta(p(u)) \d u.
    \end{e}

\begin{lemma}\label{l.sP(p^x)_critical}
Let $t\in\R_+$ and $q\in\mcl Q_2$. There is a constant $C > 0$ such that, for every $x \in \R^d$, if $p^x \in \mcl Q_{\infty,\leq 1}$ satisfies 
\begin{align}\label{e.p^x=nablapsi}
    p^x = \nabla_q \psi\Ll(q+t\nabla \xi(p^x);x\Rr),
\end{align}
then we have 
\begin{align}\label{e.l.sP(p^x)_critical}
    \Ll|\msc P_{t,q}(p;x) - \msc P_{t,q}(p^x;x)\Rr| \leq C|p-p^x|_{L^\infty}^2,\quad\forall  p \in \mcl Q_{\infty,\leq 1}.
\end{align}
\end{lemma}
\begin{proof}
For brevity of notation, we write $\psi(\scdot) = \psi(\scdot;x)$.
Let $p\in\mcl Q_{\infty,\leq 1}$ and let $B$ be given as in~\eqref{e.B=}. Then, as paths, both $p$ and $p^x$ takes value in $B$. Let  $C_1=8t^2\Ll\|\nabla \xi\lfloor_{B}\Rr\|_{\mathrm{Lip}(B;\R^{D\times D})}$, using~\eqref{e.p^x=nablapsi} and the Fréchet differentiability of $\psi$ as in~\eqref{e.frechet_bound}, we have
\begin{align*}
    \Ll|\psi(q+t \nabla\xi(p))-\psi(q+t \nabla\xi(p^x))-\la p^x , t\nabla\xi(p)-t\nabla\xi(p^x)\ra_{L^2}\Rr|
    \\
    \leq 8\Ll|t\nabla\xi(p)-t\nabla\xi(p^x)\Rr|^2_{L^2} \leq C_1|p-p^x|^2_{L^\infty}.
\end{align*} 
Recall $\theta$ from~\eqref{e.theta=}. For each $a \in\R^{D\times D}$, we interpret $\nabla^2 \xi(a)$ as the linear operator on $\R^{D\times D}$ and thus we can interpret $\nabla \theta(a) = a\cdot\nabla^2\xi(a)$ as the linear operator $\R^{D\times D}$ defined by $a'\mapsto a\cdot\nabla^2\xi(a)a'$.
Using the smoothness of $\theta$ over $B$, for some constant $C_2$ only depending on $\xi$,  we also have
\begin{align*}
    \Ll|t \int_0^1\theta(p) -t \int_0^1\theta(p^x) - \la p^x , t\nabla^2\xi(p^x)(p-p^x) \ra_{L^2}\Rr| \leq tC_2\Ll|p-p^x\Rr|^2_{L^\infty},
\end{align*}
where $\nabla^2\xi(p^x)(p-p^x)$ is interpreted pointwise as the path $[0,1]\ni s\mapsto \nabla^2\xi(p^x(s))(p(s)-p^x(s))$. Similarly, for some constant $C_3$ only depending on $\xi$, we have
\begin{align*}
    \Ll|\la p^x , t\nabla\xi(p)-t\nabla\xi(p^x)\ra_{L^2} - \la p^x , t\nabla^2\xi(p^x)(p-p^x) \ra_{L^2} \Rr|\leq t C_3\Ll|p-p^x\Rr|^2_{L^\infty}.
\end{align*}
Combining the above three displays, we get~\eqref{e.l.sP(p^x)_critical}.
\end{proof}

\section{Short-time smooth solution of the Hamilton--Jacobi equation} \label{s.short_time}

We fix $x \in \R^d$. In this section, we are interested in the partial differential equation 
\begin{align}\label{e.hj}
    \begin{cases}
        \partial_t f - \int \xi(\nabla_q f) =0 \text{ on } (0,+\infty) \times \mcl Q_2 \\
        f(0,\cdot) = \psi(\scdot;x).
    \end{cases}
\end{align}
Equation \eqref{e.hj} is a so-called Hamilton--Jacobi equation. First, we explain the exact meaning of the term $\int \xi(\nabla_q f)$. Here, $\nabla_q f$ is path-valued. More precisely, $\nabla_q f(t,q)$ is equal to some path $p:[0,1)\to\S^D$. The nonlinearity $\int \xi(\nabla_q f(t,q))$ is understood as $\int_0^1\xi(p(s))\d s$.

 We define $X(\cdot,\scdot;x) : \R_+\times \mcl Q_\infty \to L^\infty$ by 
\begin{align}\label{e.X(t,q)=}
    X(t,q;x) = q - t\nabla\xi(\nabla_q\psi(q;x)).
\end{align}
In view of Proposition~\ref{p.psi_smooth}~\eqref{i.p.psi_smooth(2)}, the definition of $X(\cdot,\scdot;x)$ can be extended to $\R_+\times \mcl Q_1$. However, we restrict the domain to $\R_+\times \mcl Q_\infty$ for better regularity.

The family curves $t \mapsto X(t,q;x)$ parametrized by $q$ ($x$ is fixed) can be understood as the characteristic lines of equation \eqref{e.hj}. More precisely, it can be checked that if $f$ is a smooth solution of \eqref{e.hj}, then 
\begin{e}
    \frac{\d}{\d t} \nabla_q f(t,X(t,q;x);x) = 0. 
\end{e}
Thanks to this, one can verify by differentiation that we must have 
\begin{e} \label{e.f along char}
    f(t,X(t,q;x);x) = \msc P_{t,q}(p;x),
\end{e}
where $p = \nabla_q f(t,X(t,q;x);x) = \nabla_q \psi(q;x)$ and $\msc P_{t,q}$ is the functional defined in \eqref{e.P_t,q(p;x)=}. If the map $q \mapsto X(t,q;x)$ is a bijection for every $t > 0$, then \eqref{e.f along char} uniquely defines a function $f$ and it is possible to verify that $f(\cdot,\scdot;x)$ is a solution of \eqref{e.hj}. 

However, at least in finite dimensional spaces, the characteristic lines can cross. So, the function $q \mapsto X(t,q;x)$ is sometimes non-injective. This means that the recipe given by \eqref{e.f along char} to define a differentiable solution may not work. Even worse, Hamilton--Jacobi equations may sometimes not have any differentiable solutions \cite[Section~3.2.5.c Example 6]{evans}. To restore the existence and uniqueness of solutions, one needs to appeal to the notion of viscosity solution \cite[Section~10]{evans}. In this setting, the notion of solution is weakened in order to also allow non-differentiable functions to solve the equation. This theory can be extended to infinite dimensional spaces, and it can be shown that \eqref{e.hj} also admits a unique viscosity solution \cite{chen2022hamilton}. 

The goal of this section is to show that on the time interval $[0,t_\star)$, the characteristic lines do not cross. This allows us to define via \eqref{e.f along char} a differentiable solution $f(\cdot,\scdot;x)$ of \eqref{e.hj} on the interval $[0,t_\star)$. In addition, by the uniqueness of the viscosity solution, the function $f(\cdot,\scdot;x)$ we build must coincide with the viscosity solution. In the end, the differentiability of the viscosity solution on the interval $[0,t_\star)$ is going to be one of the key ingredients in the verification of \eqref{e.def.phi}.

Recall the definition of the operator norm in \eqref{e.def.opnorm}. We recall that to improve readability, we sometimes abbreviate $\nabla_q$ to $\nabla$.

\begin{lemma}\label{l.invert_grad_X}
        Fix any $x\in\R^d$. The function $X(\cdot,\scdot;x) :\R_+\times \mcl Q_\infty\to L^\infty$ is Fréchet differentiable. At each $(t,q) \in \R_+\times \mcl Q_\infty$, the derivative $\nabla_q X(t,q;x)$ is a bounded linear operator from $L^\infty$ to $L^\infty$ with expression
        \begin{align}\label{e.nablaX(t,q)=1-...}
            \nabla X(t,q;x) = \mathbf{I}-t\nabla^2\xi(\nabla\psi(q;x))\nabla^2 \psi(q;x),
        \end{align}
        where $\mathbf{I}:L^\infty\to L^\infty$ is the identity operator.
        Moreover, we have
        \begin{align}\label{e.|I-nablaX|_Op<}
            \Ll\|\mathbf{I}-\nabla X(t,q;x)\Rr\|_{\mathrm{Op}(L^\infty; L^\infty)}\leq t/t_\star.
        \end{align}
        As a consequence, if $t<t_\star$, then $\nabla X(t,q;x)$ is bijective on $L^\infty$.
\end{lemma}

The interpretation of~\eqref{e.nablaX(t,q)=1-...} is as follows. For every $\kappa\in L^\infty$, $\nabla X(t,q;x)\kappa \in L^\infty$ is the path
\begin{align}\label{e.interp_nablaX(t,q)=1-...}
    [0,1]\ni s \quad\mapsto\quad  \kappa(s) - t \nabla^2\xi(p(s)) (A\kappa)(s),
\end{align}
where $p = \nabla \psi(q;x)$ is a path in $\mcl Q_{\infty,\leq 1}$ (see~\eqref{e.p.psi_smooth_1}), $\nabla^2\xi(p(s))$ is a linear operator on $\R^{D\times D}$ (as $\xi$ is a function from $\R^{D\times D}$ to $\R$), $A=\nabla^2\psi(q;x)$ is a linear operator on $L^\infty$ (see~\eqref{e.psi_smooth(4).1}), $A\kappa$ is thus a path in $L^\infty$, and lastly $\nabla^2\xi(p(s)) (A\kappa)(s)$ is the action of the linear operator $\nabla^2\xi(p(s))$ on the $\R^{D\times D}$-valued element $(A\kappa)(s)$.

\begin{proof}
Since $x \in \R^d$ is fixed, for brevity we drop the $x$ dependence in the notation and write $\psi(q)$, $X(t,q)$ instead of $\psi(q;x)$, $X(t,q;x)$. Since the dependence of $X(t,q)$ on $t$ is trivial, we only show the Fréchet differentiability in $q$ at fixed $t$.  

Let $q,q'\in \mcl Q_\infty$.  By~\eqref{e.psi_smooth(4).0} with $\sfs=\infty$, we have
\begin{align*}
    \nabla\xi\Ll(\nabla\psi(q')\Rr) = \nabla\xi\Ll(\nabla\psi(q) + \nabla^2\psi(q)(q'-q) + p_\mathsf{error}\Rr),
\end{align*}
with $\Ll|p_\mathsf{error}\Rr|_{L^\infty} \leq C_0|q'-q|^2_{L^\infty}$ for some absolute constant $C_0$.
Let $B$ be as in~\eqref{e.B=} and 
$C_1 = \Ll\|\nabla^3\xi\lfloor_{B}\Rr\|_\infty = \Ll\|\nabla^2\xi\lfloor_{B}\Rr\|_\mathrm{Lip}$, we have, for $a,a' \in B$, 
\begin{e*}
    \nabla \xi(a') =\nabla \xi(a) + \nabla^2\xi(a)(a'-a) + \mathsf{error}_a(a'-a),
\end{e*}
where $|\mathsf{error}_a(a'-a)|\leq C_1|a'-a|^2$. By~\eqref{e.p.psi_smooth_1}, $\nabla\psi(q)$ and $\nabla\psi(q')$ are paths taking values in $B$. Hence, using the above two displays, we have that,
\begin{align*}
    \nabla\xi\Ll(\nabla\psi(q')\Rr) &= \nabla\xi\Ll(\nabla\psi(q)\Rr) + \nabla^2\xi\Ll(\nabla\psi(q)\Rr)\nabla^2\psi(q)(q'-q)  \\
    &+ \nabla^2\xi\Ll(\nabla\psi(q)\Rr)p_\mathsf{error}+\mathsf{error}_{\nabla\psi(q)}\Ll(\nabla \psi(q') - \nabla \psi(q)\Rr).
\end{align*}
Again by~\eqref{e.psi_smooth(4).1}, the last two terms are bounded in the $L^\infty$-norm by $|q'-q|_{L^\infty}^2$. This shows the Fréchet differentiability of the nonlinear part in $X(t,\scdot)$ (see~\eqref{e.X(t,q)=}), which implies the differentiability of $X(t,\scdot)$ and~\eqref{e.nablaX(t,q)=1-...}. 

Next, we verify~\eqref{e.|I-nablaX|_Op<}. Let $\kappa\in L^\infty$ and recall the interpretation of~\eqref{e.interp_nablaX(t,q)=1-...}. Then, we can compute
\begin{align} \label{e.|nabla^2xi(nablapsi)nabla^2psi<}
    \Ll|\nabla^2\xi\Ll(\nabla\psi(q)\Rr)\nabla^2\psi(q)\kappa\Rr|_{L^\infty} \stackrel{\eqref{e.p.psi_smooth_1}}{\leq} \sup_{a\in B}\Ll\|\nabla^2\xi(a)\Rr\|_{\mathrm{Op}(\R^{D\times D}; \R^{D\times D})}\Ll|\nabla^2\psi(q)\kappa\Rr|_{L^\infty} \notag
    \\
    \stackrel{\eqref{e.psi_smooth(4).1}}{\leq}\sup_{a\in B}\Ll\|\nabla^2\xi(a)\Rr\|_{\mathrm{Op}(\R^{D\times D}; \R^{D\times D})}\Ll\|\nabla^2\psi(q)\Rr\|_{\mathrm{Op}(L^\infty; L^\infty)} |\kappa|_{L^\infty} \\
    \leq t^{-1}_\star|\kappa|_{L^\infty}. \notag
\end{align}
The last inequality in the display follows from \eqref{e.t_star=} once we have shown the following inequality
\begin{e}\label{e.|nabla^2psi|<|nablapsi|_Lip}
    \Ll\|\nabla^2\psi(q)\Rr\|_{\mathrm{Op}(L^\infty; L^\infty)} \leq \|\nabla\psi\|_\mathrm{Lip(\mcl Q_\infty,\mcl Q_\infty)}.
\end{e}
%
%
%

By duality between the $L^\infty$ and $L^1$ norm, the previous display is equivalent to 
\begin{e} \label{e.replacement}
    \Ll|\la \kappa', \nabla^2\psi(q_n)\kappa\ra_{L^2}\Rr| \leq |\kappa'|_{L^1} \|\nabla\psi\|_\mathrm{Lip(\mcl Q_\infty,\mcl Q_\infty)} |\kappa|_{L^\infty}, \quad \forall (\kappa,\kappa') \in L^\infty \times L^1.
\end{e}
Let us prove this inequality. Recall $\id$ is the path $s\mapsto s\Id$ and $\Id \in S^D_+$ is the $D \times D$ identity matrix. Let $\kappa : [0,1]  \to S^D$, and $\kappa \in L^\infty$. For now we assume that $\kappa$ is Lipschitz. For each $n\in\N$, let $q_n = q + n^{-1}\id$. Then, for each $n\in\N$ and any we have $q_n+\eps\kappa\in \mcl Q_\infty$ for sufficiently small $\eps$ (depending on $n$ and $\kappa$). Using the differentiability of $\psi$ in Proposition~\ref{p.psi_smooth}~\eqref{i.p.psi_smooth(4)} (with $\sfs=\infty$), for any $\kappa,\kappa'\in L^\infty$, we have
\begin{align*}
    \la \kappa', \nabla^2\psi(q_n)\kappa\ra_{L^2} = \lim_{\eps\to0}\la \kappa', \frac{\nabla\psi(q_n+\eps \kappa)-\nabla\psi(q_n)}{\eps}\ra_{L^2}.
\end{align*}
This implies,
\begin{align}\label{e.||<assume_lip}
\begin{split}
    \Ll|\la \kappa', \nabla^2\psi(q_n)\kappa \ra_{L^2}\Rr|&\leq \lim_{\eps \to 0} |\kappa'|_{L^1} \left| \frac{\nabla\psi(q_n+\eps \kappa)-\nabla\psi(q_n)}{\eps} \right|_{L^\infty} \\
    &\leq |\kappa'|_{L^1}\Ll\|\nabla\psi\Rr\|_{\mathrm{Lip}(\mcl Q_\infty,\mcl Q_\infty)} |\kappa|_{L^\infty}. 
\end{split}
\end{align}
This yields \eqref{e.replacement} but only under the stronger assumption that $\kappa$ is Lipschitz. Now let $\kappa \in L^\infty$ without assuming that $\kappa$ is Lipschitz. We let $(\kappa_k)_k$ denote a sequence of approximations of $\kappa$ defined by mollification. By usual properties of mollifiers, we have $|\kappa_k|_{L^\infty} \leq |\kappa|_{L^\infty}$ and $\kappa_k \to \kappa$ in $L^{\msf p}$ for all $\msf p \in [1,+\infty)$ (see e.g. \cite[Theorem~7 in Appendix~C.5]{evans}). Since each $\kappa_k$ is Lipschitz on $[0,1]$, as argued in~\eqref{e.||<assume_lip}, we have that 
\begin{e*} 
    \Ll|\la \kappa', \nabla^2\psi(q_n)\kappa_k\ra_{L^2}\Rr| \leq |\kappa'|_{L^1} \|\nabla\psi\|_\mathrm{Lip(\mcl Q_\infty,\mcl Q_\infty)} |\kappa_k|_{L^\infty}.
\end{e*}
In addition, since $\nabla^2 \psi(q_n)$ is a bounded operator $L^2$ according to \eqref{e.psi_smooth(4).1} as given in Proposition~\ref{p.psi_smooth}~\eqref{i.p.psi_smooth(4)}, we have that $\nabla^2 \psi(q_n) \kappa_k$ converges to $\nabla^2 \psi(q_n) \kappa$ in $L^2$ as $k\to\infty$. By extracting a subsequence, we may assume that this convergence is also pointwise a.e. on $[0,1)$.
Proposition~\ref{p.psi_smooth}~\eqref{i.p.psi_smooth(4)} also states that $\nabla^2\psi(q_n)$ is a bounded operator on $L^\infty$. Hence, these $\nabla^2 \psi(q_n) \kappa_k$ are bounded a.e.\ uniformly in $k$, which allows us to apply the dominated convergence theorem to get
\begin{e*}
    \lim_{k \to +\infty} \langle \kappa', \nabla^2 \psi(q_n) \kappa_k \rangle_{L^2}= \langle \kappa', \nabla^2 \psi(q_n) \kappa \rangle_{L^2}.
\end{e*}
Together, the previous two displays yield
\begin{e*} 
    \Ll|\la \kappa', \nabla^2\psi(q_n)\kappa\ra_{L^2}\Rr| \leq |\kappa'|_{L^1} \|\nabla\psi\|_\mathrm{Lip(\mcl Q_\infty,\mcl Q_\infty)} |\kappa|_{L^\infty}.
\end{e*}
So we have proven~\eqref{e.replacement} at $q_n\in\mcl Q_{\infty,\upa}$. Using~\eqref{e.psi_smooth(4).2} from Proposition~\ref{p.psi_smooth} (with $\sfp=\sfq=\infty$) and the convergence of $q_n$ to $q$ in $L^\infty$, we can obtain~\eqref{e.|nabla^2psi|<|nablapsi|_Lip} at $q$.

Using~\eqref{e.|nabla^2xi(nablapsi)nabla^2psi<} and~\eqref{e.nablaX(t,q)=1-...}, we obtain~\eqref{e.|I-nablaX|_Op<}. Lastly, the bijectivity is a classical consequence of~\eqref{e.|I-nablaX|_Op<} (see~\cite[Exercise~6.14]{brezis2011functional}).
\end{proof}

\begin{lemma}\label{l.Z}
Fix any $x\in\R^d$ and write $\nabla=\nabla_q$.
There is a map $Z(\cdot,\scdot;x):[0,t_\star)\times \mcl Q_\infty\to \mcl Q_\infty$ such that 
    \begin{align}\label{e.q=XZ}
        q = X(t, Z(t,q;x);x),\quad\forall (t, q)\in  [0,t_\star)\times \mcl Q_\infty.
    \end{align}
This in particular implies $X(t,\mcl Q_\infty;x)\supset \mcl Q_\infty$. Moreover, $Z(\cdot,\scdot;x):[0,t_\star)\times \mcl Q_\infty\to\mcl Q_\infty$ is Fréchet differentiable everywhere and its derivative at $(t,q)\in [0,t_\star)\times \mcl Q_\infty$ is given by
\begin{align}\label{e.def_Z=}
\begin{split}
    \nabla Z(t,q;x) &= \Ll(\nabla X(t,Z(t,q;x);x)\Rr)^{-1}, 
    \\
    \partial_t Z(t,q;x) &= \Ll(\nabla X(t,Z(t,q;x);x)\Rr)^{-1}\Ll(\nabla\xi(\nabla\psi(Z(t,q;x);x))\Rr).
\end{split}
\end{align}
\end{lemma}

Here, $\nabla X(t,Z(t,q);x)$ is the derivative of $X(t,\scdot;x)$ evaluated at $Z(t,q;x)$; $\Ll(\nabla X(t,Z(t,q;x);x)\Rr)^{-1}$ is the inverse of the operator $\nabla X(t,Z(t,q;x);x)$ given in Lemma~\ref{l.invert_grad_X}; and the last line applies this inverse to $\nabla\xi(\nabla\psi(Z(t,q;x);x))\in L^\infty$.

\begin{proof}
Again, for brevity, we omit the $x$ dependence and write $\psi$, $X(\cdot,\cdot)$, and $Z(\cdot,\cdot)$ in place of $\psi(\scdot;x)$, $X(\cdot,\scdot;x)$, and $Z(\cdot,\scdot;x)$.
Fix any $t\in[0,t_\star)$.
Let $q'\in\mcl Q_\infty$ and consider the mapping from $\Phi:\mcl Q_\infty\to\mcl Q_\infty$ given by
\begin{align*}
    \Phi(q) = q' + t\nabla\xi(\nabla\psi(q)),\quad\forall q \in \mcl Q_\infty.
\end{align*}
By~\eqref{e.p.psi_smooth_1}, this is well-defined. In view of the definition of $t_\star$ in~\eqref{e.t_star=}, we have $\|\Phi\|_\mathrm{Lip(\mcl Q_\infty;\mcl Q_\infty)}\leq t/t_\star <1$. Hence, by the fixed point iteration, there is a unique $q\in\mcl Q_\infty$ such that $q =\Phi(q)$ and thus $q'=X(t,q)$.
Hence, for each $q'\in\mcl Q_\infty$, we can define $Z(t,q') =q$. This immediately gives~\eqref{e.q=XZ}.

Next, we want to verify that $Z$ is Fréchet differentiable. For this purpose, we first verify that $Z$ is Lipschitz continuous on $[0,t_1] \times \mcl Q_\infty$ for all $t_1<t_\star$. From
\begin{align}\label{e.q=X(t,Z(t,q))=Z...}
    q\stackrel{\eqref{e.q=XZ}}{=} X(t,Z(t,q)) \stackrel{\eqref{e.X(t,q)=}}{=} Z(t,q) - t\nabla\xi(\nabla\psi(Z(t,q))),
\end{align}
and the definition of $t_\star$ in~\eqref{e.t_star=}, we can deduce
\begin{e*}
    \sup_{t\in[0,t_1]}\|Z(t,\scdot)\|_\mathrm{Lip(\mcl Q_\infty;\mcl Q_\infty)}<\infty.
\end{e*}
Similarly, from~\eqref{e.q=X(t,Z(t,q))=Z...} and the boundedness of $\nabla\psi$ as in~\eqref{e.p.psi_smooth_1}, we can deduce the Lipschitzness of $Z(\scdot,q)$ on $[0,t_1]$ uniformly in each $q\in\mcl Q_\infty$. Together, we have joint Lipschitz continuity of $Z$ on $[0,t_1]\times \mcl Q_\infty$.

Now, we are ready to prove the differentiability of $Z$.
Fix any $(t,q) \in [0,t_\star)\times \mcl Q_\infty$ and write $p=Z(t,q)$. As $(t',q')$ approaching $(t,q)$, we can use the Fréchet differentiability of $X$ from Lemma~\ref{l.invert_grad_X} and~\eqref{e.q=XZ} (at both $(t,q)$ and $(t',q')$) to see
\begin{align*}
    \Ll|q'-q - \partial_t X(t,p)(t'-t) -\nabla X(t,p)\Ll(Z(t',q')-Z(t,q)\Rr)\Rr|_{L^\infty}
    \\= o\Ll(|t'-t|+\Ll|Z(t',q')-Z(t,q)\Rr|_{L^\infty}\Rr).
\end{align*}
Notice that $\partial_t X(t,p) = -\nabla\xi(\nabla\psi(p))$ by~\eqref{e.X(t,q)=}. Solving the preceding display for $Z(t',q')-Z(t,q)$ therefore gives the positive sign in the formula for $\partial_tZ$ in~\eqref{e.def_Z=}.
By the Lipschitzness of $Z$, the error is controlled by $o\Ll(|t'-t|+|q'-q|_{L^\infty}\Rr)$. Using this and the invertibility of $\nabla X(t,p)$ given by Lemma~\ref{l.invert_grad_X}, we can see that $Z$ is Fréchet differentiable at $(t,q)$ and has its derivative given as in~\eqref{e.def_Z=}.
\end{proof}

Associated with the characteristic line $X(\cdot,\scdot;x)$ is the function $U(\cdot,\scdot;x):\R_+\times \mcl Q_\infty\to \R$ defined by, for every $(t,q) \in \R_+\times \mcl Q_\infty$,
\begin{align}\label{e.U=}
    U(t,q;x)=\psi(q;x)-t \la\nabla\xi(\nabla\psi(q;x)),\,\nabla \psi(q;x)\ra_{L^2}+ t\int\xi(\nabla\psi(q;x)).
\end{align}
Also, recall $\msc P_{t,q}(\scdot;x)$ defined in~\eqref{e.P_t,q(p;x)=}.
\begin{proposition} \label{p.f=UZ}

Fix any $x\in\R^d$. Let $f(\scdot,\scdot;x)$ be defined by
\begin{align}\label{e.f=UZ}
    f(t,q;x) = U(t,Z(t,q;x);x),\quad\forall (t,q)\in [0,t_\star)\times \mcl Q_\infty.
\end{align}
Then, $f(\scdot,\scdot;x)$ is twice Fréchet differentiable everywhere on $[0,t_\star)\times \mcl Q_\infty$ and for every $(t,q) \in [0,t_\star) \times \mcl Q_\infty$ we have  
\begin{e*}
    \partial_t f(t,q) - \int \xi(\nabla_q f(t,q)) =0.
\end{e*}
Furthermore, writing $p = \nabla_q f(t,q;x)$, we have
\begin{gather}
    p= \nabla_q \psi\big(q+t\nabla \xi (p);x\big),\label{e.fixed_pt_for_f}
    \\
    f(t,q;x) = \msc P_{t,q}\big(p; x\big).\label{e.fixed_pt_for_f_2}
\end{gather}
\end{proposition}

\begin{proof}
Again, for brevity we write $f(\scdot,\scdot)=f(\scdot,\scdot;x)$, $\psi=\psi(\scdot;x)$, and $U(\cdot,\cdot) = U(\cdot,\scdot;x)$.
We first verify the initial condition. Since $U(0,\scdot) =\psi$ and $Z(0,\scdot)$ is the identity map (due to~\eqref{e.X(t,q)=} and~\eqref{e.q=XZ}), we get $f(0,\scdot)=\psi$ as desired.

From the definition of $f$ in~\eqref{e.f=UZ}, that of $U$ in~\eqref{e.U=}, Fréchet differentiability of $Z$ in Lemma~\ref{l.Z}, and regularity of $\psi$ from Proposition~\ref{p.psi_smooth}, it is straightforward to check that $f$ is Fréchet differentiable everywhere on $[0,t_\star)\times \mcl Q_\infty$. In the following, we first verify that $f$ satisfies the equation~\eqref{e.hj}, in the process of which we will also show that $f$ is in fact twice differentiable.

For $(t,q)$ clear from the context and any relevant function $h$, we write $\nabla h^Z$ or $\partial_t h^Z$ for the evaluation of $\nabla h$ or $\partial_t h$ at the spatial variable $Z(t,q)$. For brevity, we also write
\begin{align*}
    g =\nabla\xi(\nabla\psi(\scdot)),\qquad \la\scdot,\scdot\ra=\la\scdot,\scdot\ra_{L^2}.
\end{align*}
In the following computations, we keep $(t,q)\in [0,t_\star)\times \mcl Q_\infty$ implicit.
We start with some basics:
\begin{align}
    \nabla U &\stackrel{\eqref{e.U=}}{=}  \nabla \psi - t \Ll(\nabla g\Rr)^\intercal\Ll( \nabla \psi\Rr),\notag
    \\
    \nabla X & \stackrel{\eqref{e.X(t,q)=}}{=}  \mathbf{I} - t\nabla g,\notag
    \\
    \nabla f & \stackrel{\eqref{e.f=UZ}}{=} (\nabla Z)^\intercal \nabla U^Z,\label{e.nablau=nablaZnablaU}
\end{align}
where $\mathbf{I}$ is identity operator on $L^\infty$. The first two relations give
\begin{align}\label{e.grad_U=}
    \nabla U = (\nabla X)^\intercal \nabla \psi.
\end{align}
By~\eqref{e.q=XZ}, we also have
\begin{align*}
    (\nabla Z(t,q))(\nabla X(t,Z(t,q))) = \mathbf{I},
\end{align*}
where the product on the left is the composition of two linear operators on $L^\infty$.
The above two displays along with~\eqref{e.nablau=nablaZnablaU} imply
\begin{align}
    \nabla f= \nabla \psi^Z.\label{e.nabla_u=}
\end{align}

Next, we compute the derivative of $U$ in $t$. We start with
\begin{align}
    &\partial_t X \stackrel{\eqref{e.X(t,q)=}}{=} -\nabla \xi(\nabla\psi),\label{e.X_t=}
    \\
    &\partial_t f \stackrel{\eqref{e.f=UZ}}{=} \partial_t  U^Z +\la \nabla U^Z ,\, \partial_t  Z\ra \notag
    \\
    & \stackrel{\eqref{e.U=}\eqref{e.grad_U=}}{=}-\la \nabla \xi(\nabla\psi^Z),\, \nabla \psi^Z\ra + \int\xi(\nabla \psi^Z) + \la (\nabla X^Z)^\intercal \nabla\psi^Z ,\, \partial_t Z\ra. \label{e.u_t=}
\end{align}
From~\eqref{e.q=XZ}, we get
\begin{align*}
    \partial_t X^Z + (\nabla X^Z) \partial_t Z =0.
\end{align*}
Using this and \eqref{e.X_t=} to see that the first and third terms in~\eqref{e.u_t=} cancel each other, we get
\begin{align*}
    \partial_t f = \int \xi(\nabla\psi^Z).
\end{align*}
This along with~\eqref{e.nabla_u=} implies that $f$ satisfies the equation~\eqref{e.hj} everywhere on $[0,t_\star)\times \mcl Q_\infty$ in the classical sense. Moreover, we can also infer from the above display,~\eqref{e.nabla_u=}, the differentiability of $Z$ in Lemma~\ref{l.Z}, and the regularity of $\psi$ in Proposition~\ref{p.psi_smooth} that $f$ is twice Fréchet differentiable.

Lastly, we verify~\eqref{e.fixed_pt_for_f} and~\eqref{e.fixed_pt_for_f_2}. We start with rewriting~\eqref{e.nabla_u=} as
\begin{align*}
    \nabla f(t,q) = \nabla \psi(Z(t,q)).
\end{align*}
Using~\eqref{e.X(t,q)=} and~\eqref{e.q=XZ}, we can get
\begin{align*}
    Z(t,q) = q+t\nabla \xi(\nabla\psi(Z(t,q))).
\end{align*}
Combining these two displays, we get~\eqref{e.fixed_pt_for_f}. Also, inserting these two relations into~\eqref{e.f=UZ} and recalling the definition of $U$ in~\eqref{e.U=} and $\msc P_{t,q}$ in~\eqref{e.P_t,q(p;x)=}, we can verify~\eqref{e.fixed_pt_for_f_2}. 
\end{proof}


\begin{lemma}\label{l.unique crit point} 
For every $x\in\R^d$ and $(t,q) \in [0,t_\star)\times \mcl Q_\infty$, there is a unique solution $p$ to $ p' = \nabla_q \psi (q+ t\nabla\xi(p');x)$ over $p'\in\mcl Q_\infty$.
\end{lemma}
\begin{proof}
By~\eqref{e.p.psi_smooth_1} in Proposition~\ref{p.psi_smooth}, $p\mapsto \nabla_q\psi(q+t\nabla\xi(p);x)$ is a well-defined function from $\mcl Q_\infty$ to $\mcl Q_\infty$. By the definition of $t_\star$ in~\eqref{e.t_star=}, we have
\begin{align}\label{e.pf.l.unique crit point}
    &\Ll|\nabla_q \psi (q+ t\nabla\xi(p);x)-\nabla_q \psi (q+ t\nabla\xi(p');x)\Rr|_{L^\infty} \\
    &\qquad\leq tt_\star^{-1}|p-p'|_{L^\infty},\quad\forall p,p'\in\mcl Q_\infty.
\end{align}
Then, the desired result follows from the Banach fixed point theorem.
\end{proof}

Finally, for an approximation argument later, we need the following lemma, which establishes the Lipschitz continuity of $f$ under the $L^1$ topology.

\begin{lemma}\label{l.cty_f_weak} 
Fix any $x\in\R^d$. Then, $f(\scdot,\scdot;x):[0,t_\star)\times\mcl Q_\infty\to\R$ is Lipschitz continuous with respect to the $L^1$-norm. More precisely, we have for every $(t,q),(t',q') \in [0,t_\star)\times\mcl Q_\infty$,
\begin{e}
     |f(t,q,x) - f(t',q',x)| \leq  | q-q' |_{L^1} +  |t-t'| \sup_{|a| \leq 1} |\xi(a)|.
\end{e}
In particular, this means that $f$ can be extended by continuity to a function on $[0,t_\star) \times \mcl Q_1$.
\end{lemma}

\begin{proof}
We write $f(\scdot,\scdot)=f(\scdot,\scdot;x)$, $\psi(\scdot)=\psi(\scdot;x)$, and $L^\sfr=L^\sfr([0,1);\S^D)$ for brevity. According to Proposition~\ref{p.f=UZ}, $f$ is Fréchet differentiable on $[0,t_\star) \times \mcl Q_\infty$ and we have 
\begin{align*}
    \nabla_q f(t,q) = \nabla_q \psi(q+t \nabla f(t,q))\quad \text{and}\quad
    \partial_t f(t,q) = \int \xi(\nabla_q f(t,q)).
\end{align*}
The first line together with~\eqref{e.p.psi_smooth_1} ensures that $| \nabla_q f(t,q) |_{L^\infty} \leq 1$. Plugging this in the second line yields $|\partial_t f(t,q)| \leq \sup_{|a| \leq 1} |\xi(a)| $. Thanks to this, we have 
\begin{align*}
    &|f(t,q) - f(t',q')| = \left|\int_0^1 \frac{\d}{\d s} f( s(t,q) + (1-s)(t',q')) \d s \right| \\
                        &= \Big|\int_0^1  \langle \nabla  f( s(t,q) + (1-s)(t',q')), q-q'\rangle_{L^2} \\
                        & \quad \quad \quad \quad \quad \quad + \partial_t f( s(t,q) + (1-s)(t',q'))(t-t') \d s \Big| \\
                        &\leq   | q-q' |_{L^1} \sup_s | \nabla  f( s(t,q) + (1-s)(t',q')) |_{L^\infty} +  |t-t'|  \sup_{|a| \leq 1} |\xi(a)|  \\
                        &\leq   | q-q' |_{L^1} +  |t-t'| \sup_{|a| \leq 1} |\xi(a)|.
\end{align*}
\end{proof}

\section{Proof of the main results} \label{s.proof_main_results}


\subsection{Preliminaries}

From~\cite[Proposition~3.1]{chenmourrat2023cavity}, we recall that $\bar F_N^{m \mapsto x \cdot m}$ defined in~\eqref{e.F_N(t,q;s,x)=} is Lipschitz in $(t,q)$ uniformly in $(N,x)$. More precisely, for every $t,t' \in \R_+$ and $q,q' \in \mcl Q_1$, we have 
\begin{e}\label{e.F_NLip}
    \Ll|\bar F_N^{m \mapsto x \cdot m}(t,q) - \bar F_N^{m \mapsto x \cdot m}(t',q')\Rr| \leq \Ll|q-q'\Rr|_{L^1} + \Ll|t-t'\Rr| \sup_{|a| \leq 1} |\xi(a)|.
\end{e}
According to the Arzelà--Ascoli theorem, the sequence $(\bar F_N^{m \mapsto x \cdot m})_{N \in \N}$ admits subsequential limits in the topology of local uniform convergence on $\R_+\times \mcl Q_\sfr$ for any $\sfr \in[1,\infty]$. 


Recall the definition of $\msc P_{t,q}$ from~\eqref{e.P_t,q(p;x)=}. The following proposition is partly inspired by~\cite[Proposition~1.3]{dominguez2024critical}. 

\begin{proposition}\label{p.id_limF}
Let $x\in\R^d$. Then, for every $(t,q)\in[0,t_\star)\times\mcl Q_2$, we have
\begin{align*}
    \lim_{N\to\infty} \bar F^{m \mapsto x \cdot m}_N(t,q) = \msc P_{t,q}(p;x),
\end{align*}
where $p$ is the unique solution to $p=\nabla_q\psi(q+t\nabla \xi(p);x)$ given by Lemma~\ref{l.unique crit point}.
\end{proposition}

To prove this proposition, we verify that any subsequential limit of the sequence $(\bar F^{m \mapsto x \cdot m}_N(\cdot,q))_N$ is differentiable on $(0,t_\star)$ and that its derivative can be uniquely characterized. This will allow us to show that the sequence $(\bar F^{m \mapsto x \cdot m}_N(\cdot,q))_N$ converges and to identify its limit.
To proceed, we will need the following lemma, which states that the free energy $\bar F^{m \mapsto x \cdot m}_N(\cdot,q)$ is locally semi-concave in $t\in(0,+\infty)$, uniformly in $N$, $q$, and $x$.

\begin{lemma}\label{l.semi-concave_in_t}
For every $c>0$ there is a constant $C>0$ such that for every $N\in\N$, every $x\in\R^d$ and every $q\in\mcl Q_1$ we have for all $t,t' \in [c,+\infty)$ and $\lambda\in[0,1]$, 
\begin{align*}
    (1-\lambda)\bar F_N(t)+ \lambda\bar F_N(t')- \bar F_N((1-\lambda)t+\lambda t')
    \leq C\lambda(1-\lambda)(t-t')^2
\end{align*}
where we used the shorthand $\bar F_N(\cdot) = \bar F^{m \mapsto x \cdot m}_N(\cdot,q)$.
\end{lemma}

\begin{proof}
Recall the Hamiltonian $H^G_{N;t,q}(\sigma,\alpha)$ from~\eqref{e.H^t,q;s,x_N(sigma,alpha,chi)=}. For each $s\geq 0$, we set $H_s(\sigma,\alpha) = H^{m \mapsto x \cdot m}_{N;s^2,q}(\sigma,\alpha)$ and define
\begin{align*}
    G(s) = - \frac{1}{N} \E\log \iint \exp\Ll(H_s(\sigma,\alpha) \Rr) \d P_1^{\otimes N}(\sigma) \d \mfk R(\alpha).
\end{align*}
In view of~\eqref{e.F_N(t,q;s,x)=}, we have $\bar F_N(t) = G(\sqrt{t})$. This implies,
\begin{e*}
   \bar F_N''(t) = \frac{G''(\sqrt{t})}{4t} - \frac{G'(\sqrt{t})}{4t^{3/2}}.
\end{e*}
The local semi-concavity of $\bar F_N$ can be reformulated as $\bar F_N''$ being upper-bounded on $[c,+\infty)$. Gaussian integration by parts gives a uniform bound on $G'$, so it remains to bound $G''$ from above. 

Let $R=\lambda (1-\lambda)(s-s')^2 N\xi(\sigma\sigma^\intercal /N)$. Using the definition of $G$, we have
\begin{align*}
    &(1-\lambda)G(s) + \lambda G(s') = -\frac{1}{N} \E\log \Ll(\cdots\Rr)^{1-\lambda}(\cdots)^\lambda
    \\
    &\stackrel{\text{Jensen}}{\leq}-\frac{1}{N}\E\log \iint e^{(1-\lambda) H_s(\sigma,\alpha) + \lambda H_{s'}(\sigma,\alpha)}  \d P_1^{\otimes N}(\sigma) \d \mfk R(\alpha)
    \\
    &= -\frac{1}{N}\E\log \iint e^{H_{(1-\lambda)s + \lambda s'}(\sigma,\alpha)  - R}  \d P_1^{\otimes N}(\sigma) \d \mfk R(\alpha).
\end{align*}
Setting $M= \max_{a\in\S^D_+:\:|a|\leq 1}|\xi(a)|$, we have $R \leq \lambda(1-\lambda)(s-s')^2 N M$ and thus
\begin{align*}
    (1-\lambda)G(s) + \lambda G(s') \leq G((1-\lambda)s + \lambda s') + \lambda(1-\lambda)(s-s')^2M.
\end{align*}
Setting $\lambda = \frac{1}{2}$ and $s'=s+\eps$ and sending $\eps\to0$ in this display, we get $G''(s)\leq 2M$, completing the proof as explained before.
\end{proof}

\begin{proof}[Proof of Proposition~\ref{p.id_limF}]
Since $x$ is fixed, we omit it from the notation and we make the dependence of the fixed point $p$ in $(t,q)$ explicit by writing it $p^{t,q}$. 

The Arzelà--Ascoli theorem gives a subsequence along which the free energy converges to some limit $f:\R_+\times \mcl Q_2\to\R$. We recall from~\cite[Proposition~5.3]{chenmourrat2023cavity} that $f$ is Gateaux differentiable jointly in $t$ and $q$ on a dense subset of $\R_+\times \mcl Q_2$ (in the $L^2$-topology). We let $G$ denote the set of Gateaux differentiability points of $f$. For each $(t,q)\in G$, \cite[Proposition~7.2]{chenmourrat2023cavity} gives $\partial_t f(t,q) = \int_0^1\xi(p'^{t,q})$ for some $p'^{t,q}\in\mcl Q_{\infty,\leq 1}$ that satisfies the fixed point equation $p'=\nabla_q\psi(q+t\nabla\xi(p'))$. Lemma~\ref{l.unique crit point} implies $p'^{t,q}=p^{t,q}$ whenever $t<t_\star$ and thus
\begin{align*}
    \partial_t f(t,q) = \int_0^1\xi(p^{t,q}),\quad\forall (t,q)\in G ,\  t<t_\star.
\end{align*}

Now, fix any $(t,q)\in (0,t_\star)\times\mcl Q_2$. For $s$ sufficiently small, there is $(t_s,q_s)\in G$ satisfying $|t-t_s|+|q-q_s|_{L^2}\leq s^2$. From this and the Lipschitz continuity of $f$ inherited from~\eqref{e.F_NLip}, we have that there exists some absolute constants $C_1, C_2>0$ such that for sufficiently small $s$,
\begin{align*}
    s^{-1}\Ll(f(t+s,q) - f(t,q)\Rr) &\leq s^{-1}  \Ll(f(t_s+s,q_s) - f(t_s,q_s)\Rr) + C_1s
    \\
    &\leq \int_0^1\xi(p^{t_s,q_s}) + C_2s,
\end{align*}
where we used the local semi-concavity of $f$ inherited from that enjoyed by the free energy as given in Lemma~\ref{l.semi-concave_in_t} above.
For the converse bound, we choose $(t'_s,q'_s)$ satisfying $(t'_s+s,q'_s)\in G$ and $|t-t'_s|+|q-q'_s|_{L^2}\leq s^2$. By similar arguments, we get that there exists some absolute constants $C_3, C_4>0$, such that for sufficiently small $s$,
\begin{align*}
    s^{-1}\Ll(f(t+s,q) - f(t,q)\Rr) &\geq s^{-1}  \Ll(f(t'_s+s,q'_s) - f(t'_s,q'_s)\Rr) - C_3s
    \\
    &\geq \int_0^1\xi(p^{t'_s+s,q'_s})- C_4s.
\end{align*}
By a similar argument as in~\eqref{e.t_star_used_here_2}, we can see that the map $[0,t_\star)\times \mcl Q_2\ni (t,q)\mapsto p^{t,q}$ is continuous in $L^\infty$. Therefore, sending $s\to0$ in the above two displays, we obtain the differentiability of $f(\cdot,q)$ and
\begin{align*}
    \partial_t f(t,q) = \int_0^1\xi(p^{t,q}),\quad\forall (t,q)\in (0,t_\star)\times\mcl Q_2.
\end{align*}
Hence, for every $(t,q)\in (0,t_\star)\times\mcl Q_2$, we conclude
\begin{align*}
    f(t,q) =\psi(q)+ \int_0^t\Big(\int_0^1\xi(p^{s,q}(u))\d u\Big) \d s.
\end{align*}
This uniquely characterizes $f$ and proves that it is independent of the chosen subsequence. Therefore, we conclude that $\bar F_N$ converges at every $(t,q)\in (0,t_\star)\times\mcl Q_2$. By~\cite[Proposition~7.3]{chenmourrat2023cavity}, the limit at $(t,q)$ is equal to $\msc P_{t,q}(p;x)$ for some $p\in\mcl Q_{\infty,\leq 1}$ satisfying the fixed point equation $p=\nabla_q\psi(q+t\nabla\xi(p);x)$. Again, Lemma~\ref{l.unique crit point} ensures that $p=p^{t,q}$, completing the proof.
\end{proof}

\begin{corollary}\label{c.limF_N=f}
Fix any $x\in\R^d$. Then, for every $(t,q) \in [0,t_\star)\times \mcl Q_\infty$, we have
\begin{align}
    \lim_{N\to\infty} \bar F^{m \mapsto x \cdot m}_N(t,q) = f(t,q;x),
\end{align}
where $f(\scdot,\scdot;x)$ is the function constructed in Proposition~\ref{p.f=UZ}.
\end{corollary}

\begin{proof}
This follows from Proposition~\ref{p.id_limF} and Proposition~\ref{p.f=UZ}.
\end{proof}

\begin{lemma}\label{l.varphi}
Fix any $(t,q)\in[0,t_\star)\times \mcl Q_\infty$ the function $x \mapsto f(t,q;x)$ is Lipschitz, concave, and continuously differentiable. 
\end{lemma}


\begin{proof}
Since we have $f(t,q;x) = \lim_{N \to +\infty} \bar F^{m \mapsto x \cdot m}_N(t,q)$ for every $x\in\R^d$ due to Corollary~\ref{c.limF_N=f}, the concavity and Lipschitzness of $f(t,q;\cdot)$ follows from that of $x \mapsto \bar{F}^{m \mapsto x \cdot m}_N(t,q)$, which are evident from its definition in~\eqref{e.F_N(t,q;s,x)=} and the boundedness of $m_N$ (see~\eqref{e.m_N=}). 
It remains to prove the differentiability. Once this is done, the continuity of the derivative follows directly since $f(t,q;\cdot)$ is concave.
Recall $\msc P_{t,q}$ defined in~\eqref{e.P_t,q(p;x)=}. Since $t$ and $q$ are fixed, we write $\msc P(\scdot;x)= \msc P_{t,q}(\scdot;x)$ for brevity.
By Proposition~\ref{p.f=UZ}, $\nabla_q f(t,q;x)$ exists for every $x\in \R^d$ and we denote it by $p^x = \nabla_q f(t,q;x)$. 
By~\eqref{e.p.psi_smooth_1}, we have $p^x\in\mcl Q_{\infty,\leq1}$ defined in~\eqref{e.Q_infty,<1}.
By~\eqref{e.fixed_pt_for_f} and~\eqref{e.fixed_pt_for_f_2} from the same proposition, we have
\begin{align*}
    f(t,q;x) = \msc P(p^x;x)\quad\text{and}\quad p^x = \nabla_q \psi\Ll(q+t\nabla\xi (p^x);x \Rr),\quad\forall x\in\R^d.
\end{align*}
Using the latter relation, for $x,x'\in\R^d$ we can compute 
\begin{align}
    |p^x - p^{x'}|_{L^{\infty}}\leq \Ll|\nabla_q \psi\Ll(q+t\nabla\xi (p^x);x \Rr)-\nabla_q \psi\Ll(q+t\nabla\xi (p^{x'});x \Rr)\Rr|_{L^\infty} \notag
    \\
    +  \Ll|\nabla_q \psi\Ll(q+t\nabla\xi (p^{x'});x \Rr)-\nabla_q \psi\Ll(q+t\nabla\xi (p^{x'});x' \Rr)\Rr|_{L^\infty} \notag
    \\
    \stackrel{\eqref{e.t_star=}\eqref{e.p.psi_smooth_2}}{\leq} t t_\star^{-1} |p^x - p^{x'}|_{L^\infty} + 4|h|_{L^\infty}|x-x'|. \label{e.t_star_used_here_2}
\end{align}
Due to $t<t_\star$, we can see that $x\mapsto p^x$ is Lipschitz from $\R^d$ to $\mcl Q_{\infty,\leq 1}$. 
Using~\eqref{e.smooth(3).1.5} from Proposition~\ref{p.psi_smooth}~\eqref{i.p.psi_smooth(3)}, we have
\begin{align*}
    &\Ll|\msc P(p^{x'};x')- \msc P(p^{x'};x)
    - (x'-x)\cdot \nabla_x \psi(q+t\nabla\xi(p^{x'});x)\Rr| \\
    &\qquad\leq |h|^2_{L^\infty}\Ll|x-x'\Rr|^2.
\end{align*}
The Lipschitzness of $x\mapsto p^x$ further implies that as $x'$ tends to $x$, we have
\begin{align*}
    \msc P(p^{x'};x')- \msc P(p^{x'};x) = (x'-x)\cdot \nabla_x \psi(q+t\nabla\xi(p^x);x)+o(|x-x'|).
\end{align*}
Furthermore, using~\eqref{e.l.sP(p^x)_critical} from Lemma~\ref{l.sP(p^x)_critical}, we have
\begin{e*}
   |\msc P(p^{x'};x)- \msc P(p^x;x)| \leq C|p^{x'}-p^x|_{L^\infty}^2.
\end{e*}
Combining the previous two displays; we obtain 
\begin{align*}
    \msc P(p^{x'};x') - \msc P(p^x;x) = \msc P(p^{x'};x') - \msc P(p^{x'};x) + \msc P(p^{x'};x) - \msc P(p^x;x)\\
    = (x'-x)\cdot \nabla_x \psi(q+t\nabla\xi(p^x);x) + o(|x-x'|) + o\Ll(|p^x-p^{x'}|_{L^\infty}\Rr).
\end{align*}
Observe that the combined error is $o(|x-x'|)$ due to the Lipschitzness of $x\mapsto p^x$ proved above. Recall $\msc P(p^x;x) =f(t,q;x)$ and we conclude from the above display that $f(t,q;\cdot)$ is differentiable at every $x$.
Since $f(t,q;\cdot)$ is also concave, then it is a classical result from convex analysis that differentiability implies continuous differentiability.
\end{proof}

\subsection{Limit free energy and the large deviation principle}

Recall the definition of the Gibbs measure $\langle \cdot \rangle_{N;t,q}^G$ from \eqref{e.<>_Ntqsx=} and recall that given a random variable $X : \Omega \to \mfk X$ and $\omega \in \Omega$ we let $X^\omega \in \mfk X$ denote the realization of the random variable $X$ on the event $\omega$. Also, recall the definition of $f^*(t,q;m)$ from~\eqref{e.f^*=}.

\begin{lemma} \label{l.ldp_G=0}
    Let $(t,q) \in [0,t_\star) \times \mcl Q_\infty$. There exists a full-measure subset $\Omega' \subset \Omega$ such that, for every $\omega \in \Omega'$, $m_N$ under $\langle \cdot \rangle_{N;t,q}^{0,\omega}$ satisfies a large deviation principle with rate function $I_{t,q}$ defined by 
    \begin{e}\label{e.I_t,q=f^*-f}
        I_{t,q}(m)  = f^*(t,q;m) - f(t,q;0), \qquad \forall m \, \in \R^d.
    \end{e}
\end{lemma}

\begin{proof}
Using Corollary~\ref{c.limF_N=f} and the concentration of free energy as in Remark~\ref{r.concentration}, we can find a full-measure subset $\Omega' \subset \Omega$ such that, for every $\omega\in\Omega'$ and every $y\in \Q^d$,
\begin{equation*}
    \lim_{N\to\infty}F^{m \mapsto y \cdot m, \omega}_N(t,q)= f(t,q;y).
\end{equation*}
Henceforth, we fix $\omega\in\Omega'$ and define for each $y\in\R^d$,
\begin{e*}
    \Lambda_{N;t,q}(y)= \frac{1}{N}\log \la e^{Ny\cdot m_N}\ra^{0,\omega}_{N;t,q}.
\end{e*}
Notice that $\Lambda_{N;t,q}(y) = -F_N^{m \mapsto y \cdot m,\omega}(t,q) + F_N^{0,\omega}(t,q)$. Hence, we have  for every $y\in \Q^d$
\begin{e}\label{e.limLambda_N=}
    \lim_{N\to\infty}\Lambda_{N;t,q}(y)= -f(t,q;y)+f(t,q;0).
\end{e}
Since $\Lambda_{N;t,q}$ is Lipschitz uniformly in $N$ and $f(t,q;\cdot)$ is also Lipschitz, the previous display in fact holds for every $y \in \R^d$. Finally, since by Lemma~\ref{l.varphi}, $f(t,q;\cdot)$ is differentiable everywhere, the Gärtner--Ellis theorem (e.g., see~\cite[Theorem~2.3.6]{dembo2009large}) guarantees that $m_N$ satisfies a large deviation principle under $\la\cdot\ra^{0,\omega}_{N;t,q}$ with rate function 
\begin{e*}
    I_{t,q}(m) = \sup_{y \in \R^d} \left\{ y \cdot m - \lim_{N \to +\infty} \Lambda_{N;t,q}(y) \right\} \stackrel{\eqref{e.limLambda_N=}\eqref{e.f^*=}}{=}  f^*(t,q;m) - f(t,q;0).
\end{e*}
\end{proof}

\begin{lemma} \label{l.log<e^G>}
    Let $(t,q) \in [0,t_\star) \times \mcl Q_\infty$. For every continuous function $G : \R^d \to \R$ and every $\omega \in \Omega'$ (as given in Lemma~\ref{l.ldp_G=0}), we have 
    \begin{e}
        \lim_{N \to +\infty} \frac{1}{N} \log \la e^{NG(m_N)} \ra_{N;t,q}^{0,\omega} = \sup_{m \in \R^d} \left\{ G(m) - I_{t,q}(m) \right\}.
    \end{e}
\end{lemma}

\begin{proof}
Since $m_N$ is bounded uniformly in $N$, so is $G(m_N)$, which allows us to apply Varadhan's lemma (e.g., see~\cite[Theorem~4.3.1]{dembo2009large}) to deduce the desired result from the large deviation principle in Lemma~\ref{l.ldp_G=0}.
\end{proof}

\begin{proof}[Proof of Theorem~\ref{t.free_energy}]
Let $\Omega'$ be given as in Lemma~\ref{l.ldp_G=0}.
Notice that, for every $\omega\in\Omega'$, 
\begin{e*}
    \frac{1}{N}\log\la e^{NG(m_N)}\ra^{0,\omega}_{N;t,q}= -F^{G,\omega}_N(t,q) + F^{0,\omega}_N(t,q).
\end{e*}
Using Lemma~\ref{l.log<e^G>}, \eqref{e.I_t,q=f^*-f}, and $\lim_{N\to\infty} F^{0,\omega}_N(t,q)=f(t,q;0)$ due to Corollary~\ref{c.limF_N=f}, we get that for every $\omega \in \Omega'$
\begin{e*}
    \lim_{N\to\infty} F^{G,\omega}_N = \inf_{m \in \R^d} \left\{ -G(m) + f^*(t,q;m) \right\}.
\end{e*}
Finally, using the definition of $f^*(t,q;m)$ in~\eqref{e.f^*=}, we observe that the right-hand side in the previous display is equal to the right-hand side in~\eqref{e.free_energy}. Then, the desired result follows from this and the concentration of free energy as in Remark~\ref{r.concentration}.
\end{proof}

\begin{proof}[Proof of Theorem~\ref{t.ldp}]
Let $\Omega'$ be given as in Lemma~\ref{l.ldp_G=0}. Fix any continuous function $G: \R^d \to \R$ and recall $I^G_{t,q}$ from~\eqref{e.I^G=}. Let $\Gamma:\R^d\to\R$ be any bounded continuous function. Since we have, 
\begin{e*}
    \frac{1}{N}\log\la e^{N\Gamma(m_N)}\ra^{G,\omega}_{N;t,q} = \frac{1}{N}\log\la e^{N(G+\Gamma)(m_N)}\ra^{0,\omega}_{N;t,q} - \frac{1}{N} \log\la e^{NG(m_N)}\ra^{0,\omega}_{N;t,q}.
\end{e*}
It follows from Lemma~\ref{l.log<e^G>} along with~\eqref{e.I_t,q=f^*-f} that for every $\omega \in \Omega'$,
\begin{align*}
    \lim_{N\to\infty}\frac{1}{N}\log\la e^{N\Gamma(m_N)}\ra^{G,\omega}_{N;t,q} &=\sup\Ll\{G+ \Gamma-I_{t,q} \Rr\} - \sup\Ll\{G-I_{t,q} \Rr\} \\
    & \stackrel{\eqref{e.I^G=}\eqref{e.I_t,q=f^*-f}}{=} \sup \{\Gamma - I^G_{t,q}\}.
\end{align*}
Since this holds for any bounded continuous $\Gamma$, we can apply a variant of Bryc's inverse Varadhan lemma (e.g., see the ``$\Longleftarrow$'' direction in~\cite[Theorem~4.4.13]{dembo2009large}) which guarantees that $m_N$ satisfies a large deviation principle under $\la\cdot\ra^{G,\omega}_{N;t,q}$ with rate function $I^G_{t,q}$.
\end{proof}


\subsection{Replica symmetric form} \label{ss.proj}

Recall from~\eqref{e.haty=} the notation of a constant path $\hat y$ given any $y\in\S^D$. In this section, we also write $(y)^\wedge = \hat y$. For $\kappa\in L^1$, we define $\check{\kappa}=(\kappa)^\vee =\int_0^1\kappa(r)\d r \in\S^D$. Then, for every $y\in \S^D$ and $\kappa\in L^1$, we have the following simple identities:
\begin{align}\label{e.proj_adjoint}
    y\cdot \check{\kappa} = \la \hat y, \kappa\ra_{L^2}\quad\text{and}\quad \Ll(\hat y\Rr)^\vee = y. 
\end{align}
Recall $\psi(\scdot,\scdot)$ in~\eqref{e.psi(q;x)=} and $\phi(\scdot,\scdot)$ in~\eqref{e.phi(y,x)=}. Also, recall the relation $\phi(y;x)=\psi\Ll(\hat y;x\Rr)$ from~\eqref{e.phi=psi}.
To prove Theorem~\ref{t.rs}, we show in Proposition~\ref{p.f-d-proj} that for every $x \in \R^D$, the function $(t,y) \mapsto f(t,\hat{y};x)$ satisfies a \emph{finite} dimensional Hamilton--Jacobi on $[0,t_\star) \times \S^D_+$. To do so, we need to recall a modified version of~\cite[Lemma~5.12]{chen2024simultaneous}. This lemma relies on the assumption that the support of $P_1$ spans $\R^D$.

\begin{lemma}[\cite{chen2024simultaneous}]\label{l.detect_jump}
    Fix any $x\in\R^d$. Let $q\in\mcl Q_\infty$ and set $p=\nabla_q\psi(q;x)$. For $0<s<s'<1$, we have $p(s')\neq p(s)$ if and only if $q(s')\neq q(s)$.
\end{lemma}

In~\cite{chen2024simultaneous}, the paths (such as $p,q$) are taken to be their left-continuous versions. But the result remains the same here for right-continuous ones (see the definition of $\mcl Q$ in Section~\ref{ss.setting}). The setting in~\cite{chen2024conventional} concerns possibly non-convex vector spin glass models without extra random spins and Curie--Weiss-type interactions. Namely, the function $G$ is chosen to be the null function in~\eqref{e.H^t,q;s,x_N(sigma,alpha,chi)=} and~\eqref{e.F_N(t,q;s,x)=}. We explain that the result in~\cite{chen2024conventional} can be adapted straightforwardly.

\begin{proof}[Sketch of Proof of Lemma~\ref{l.detect_jump}]    
    The proof of Lemma~\ref{l.detect_jump} and most of the other results in~\cite{chen2024conventional} are based on the analysis of the Parisi PDE and the associated SDE. The Parisi PDE gives a concrete way to compute the integration w.r.t.\ the Gaussian field $(w^q(\alpha))_\alpha$ indexed by the cascade variable $\alpha$ in $\psi(q;0)$ (see~\eqref{e.psi(q;x)=}). Now, with the additional field $x\cdot h(\tau,\chi_1)$, the initial condition (see~\cite[(4.2)]{chen2024simultaneous}) of the Parisi PDE becomes the function
    \begin{align*}
        \R^D \ni z\quad\mapsto\quad  \E_{\chi_1}\log \int \exp\Ll(\tau\cdot z-\tau\tau^\intercal\cdot a+x\cdot h(\tau,\chi_1)\Rr)\d P_1(\tau),
    \end{align*}
    where $a\in\S^D_+$ is some constant. The version in~\cite[(4.2)]{chen2024simultaneous} corresponds to the case $x=0$. The analysis in~\cite{chen2024simultaneous} only uses the smoothness, Lipschitzness, and convexity of the initial condition, which are still preserved despite the presence of $h(\tau,\chi_1)$. Therefore, working with the above initial condition, we can straightforwardly adapt~\cite[Lemma~5.12]{chen2024simultaneous} to Lemma~\ref{l.detect_jump} here.
\end{proof}

\begin{proposition}\label{p.f-d-proj}
    Let $x\in\R^d$ and let $f(\scdot,\scdot;x)$ be given as in Proposition~\ref{p.f=UZ}. Define $g(\scdot,\scdot;x):[0,t_\star)\times \S^D_+\to\R$ by
    \begin{e}
        g(t,y;x)=f(t,\hat y;x).
    \end{e}
    The function $g(\scdot,\scdot;x)$ is differentiable on $[0,t_\star)\times \S^D_+$ and satisfies
    \begin{align}\label{e.hj_g-p.f-d-proj}
        \begin{cases}
            \partial_t g(t,y;x) - \xi\big(\nabla_y g(t,y;x)\big)=0 &\quad \text{ for } (t,y)\in(0,t_\star)\times \S^D_+, \\
            g(0,y;x) = \phi(y;x) &\quad \text{ for } y \in \S^D_+. 
        \end{cases}
    \end{align}
    Moreover, at every $(t,y) \in [0,t_\star)\times \S^D_+$ we have,
    \begin{align}\label{e.der_g-p.f-d-proj}
        \partial_t g(t,y;x) = \partial_t f(t,\hat y;x)\quad\text{and}\quad \Ll(\nabla_y g(t,y;x)\Rr)^\wedge= \nabla_qf(t,\hat y;x).
    \end{align}
    Finally, setting $z=\nabla_yg(t,y;x)$, we have that $z$ is the unique solution to $z'= \nabla_y \phi(y+t\nabla\xi(z');x)$ over $z'\in\S^D_+$ and we have
    \begin{e}\label{e.fixed_pt_eqn_fd}
        g(t,y;x) = \phi(y+t\nabla\xi(z);x) - t \theta(z).
    \end{e}
\end{proposition}

\begin{proof}
For brevity, we drop $x$ from the notation and write $g=g(\scdot,\scdot;x)$, $f=f(\scdot,\scdot;x)$, and $\psi=\psi(\scdot;x)$.
Fix any $(t,y) \in [0,t_\star)\times \S^D_+$.
From the definition of $g$, it is easy to see the first relation in~\eqref{e.der_g-p.f-d-proj}. To see the second one, let $y'\in\S^D_+$ and we have
\begin{align*}
    \frac{\d}{\d\eps}g(t,y+\eps(y'-y))\big|_{\eps=0} = \la (y'-y)^\wedge,\nabla_q f(t,\hat y) \ra_{L^2} 
    \\\stackrel{\eqref{e.proj_adjoint}}{=}(y'-y)\cdot \Ll(\nabla_q f(t,\hat y)\Rr)^\vee.
\end{align*}
Hence, we can deduce the differentiability of $g$ in $y$ and
\begin{align}\label{e.nabla_yg(t,y)=()^v}
    \nabla_y g(t,y) = \Ll(\nabla_q f(t,\hat y)\Rr)^\vee.
\end{align}
Using the definition of $t_\star$ in~\eqref{e.t_star=}, we can compute, for any $z,z'\in\S^D_+$,
\begin{align}\label{e.t_star_used_here}
    &\Ll|\Ll(\nabla_q\psi(\hat y+t\nabla\xi(\hat z))\Rr)^\vee
    - \Ll(\nabla_q\psi(\hat y+t\nabla\xi(\hat z'))\Rr)^\vee\Rr| \\
    &\qquad\leq tt_\star^{-1}\Ll|\hat z-\hat z'\Rr|_{L^\infty}
    = tt_\star^{-1}\Ll|z- z'\Rr|.
\end{align}
Due to $t<t_\star$, by the Banach fixed point theorem, there is a unique $z\in\S^D_+$ such that
\begin{align}\label{e.z=(...)^v}
    z = \Ll(\nabla_q\psi(\hat y+t\nabla\xi(\hat z))\Rr)^\vee.
\end{align}
Since $\hat y+t\nabla\xi(\hat z)$ is a constant path, Lemma~\ref{l.detect_jump} implies that $\nabla_q\psi(\hat y+t\nabla\xi(\hat z))$ is also a constant path. Therefore, the above display yields
\begin{align*}
    \hat z = \nabla_q\psi(\hat y+t\nabla\xi(\hat z)).
\end{align*}

On the other hand, \eqref{e.fixed_pt_for_f} in Proposition~\ref{p.f=UZ} together with Lemma~\ref{l.unique crit point} implies that $\nabla_q f(t,\hat y)$ is the unique solution in $\mcl Q_2$ to $p=\nabla_q\psi(\hat y+t\nabla\xi(p))$. This, along with the above display, gives $\hat z= \nabla_q f(t,\hat y)$. Hence, $\nabla_q f(t,\hat y)$ is a constant path and thus $\Ll(\Ll(\nabla_q f(t,\hat y)\Rr)^\vee\Rr)^\wedge = \nabla_q f(t,\hat y)$, which together with~\eqref{e.nabla_yg(t,y)=()^v} gives the second relation in~\eqref{e.der_g-p.f-d-proj}. Using these, we also have $\int \xi(\nabla_qf(t,\hat y))= \xi(\nabla_yg(t,y))$. This along with~\eqref{e.phi=psi} that $f$ satisfies the equation~\eqref{e.hj} implies that $g$ satisfies the equation~\eqref{e.hj_g-p.f-d-proj}.

Lastly, we turn to~\eqref{e.fixed_pt_eqn_fd} and the uniqueness statement above~\eqref{e.fixed_pt_eqn_fd}.
Notice that $(y+t\nabla\xi(z))^\wedge=\hat y + t\nabla\xi(\hat z)$.
Due to~\eqref{e.phi=psi}, by the same argument for~\eqref{e.nabla_yg(t,y)=()^v}, we have $\nabla_y\phi(y+t\nabla\xi(z))=\Ll(\nabla_q\psi(\hat y+t\nabla\xi(\hat z))\Rr)^\vee$, which along with~\eqref{e.z=(...)^v} gives $z= \nabla_y \phi(y+t\nabla\xi(z);x)$.
Using~\eqref{e.t_star_used_here}, we can also see that $z$ is the unique solution of this relation as claimed. Since we have shown $\hat z = \nabla_qf(t,\hat y)$, by~\eqref{e.fixed_pt_for_f_2} (see $\msc P_{t,q}$ in~\eqref{e.P_t,q(p;x)=}), we have $g(t,y)=f(t,\hat y)=\psi(\hat y+ t\nabla\xi(\hat z))-t\theta(z)$, which along with~\eqref{e.phi=psi} gives~\eqref{e.fixed_pt_eqn_fd}.
\end{proof}


\begin{proof}[Proof of Theorem~\ref{t.rs}]
The theorem follows from Proposition~\ref{p.f-d-proj}.
\end{proof}


\appendix

\section{Multi-species models}\label{s.multi-sp}

We adapt the main results to multi-species spin glass models. 

\subsection{Setting}


We fix a finite set $\sS$ containing symbols for different species.
For each $N\in\N$, let $(I_{N,\s})_{\s\in\sS}$ be a partition of $\{1,\dots,N\}$. Each $I_{N,\s}$ contains indices belonging to the $\s$-species. 
For each $N\in\N$, we define
\begin{align}\label{e.lambda_N,s=}
    \lambda_{N,\s} = |I_{N,\s}|/N,\quad\forall \s\in\sS;\qquad \lambda_N=\Ll(\lambda_{N,\s}\Rr)_{\s\in\sS}.
\end{align}
For each $N \in \N\cup\{\infty\}$, we consider the simplex
\begin{align*}
    \blacktriangle_N = \Big\{(\lambda_\s)_{\s\in\sS}\ \big|\ \lambda_\s \in [0,1]\cap \left(\frac{1}{N}\Z \right),\,\forall \s\in\sS;\; \sum_{\s\in\sS}\lambda_\s =1\Big\}.
\end{align*}
Here $\frac{1}{N}\Z = \{ k/N : k \in \Z \}$, with the understanding that $\Z/\infty = \R$ when $N=\infty$. In view of~\eqref{e.lambda_N,s=}, we have $\lambda_N\in \blacktriangle_N$ for each $N\in\N$. Throughout this appendix, we assume that
\begin{align}\label{e.limlambda_N=lambda_infty}
    \lim_{N\to\infty} \lambda_N = \lambda_\infty,
\end{align}
for some $\lambda_\infty\in \blacktriangle_\infty$.

For each $\s\in\sS$, let $P_\s$ be a finite positive measure supported on $[-1,+1]$ and we assume $\supp P_\s \supsetneqq\{0\}$ (related to the fact we require the support of $P_1$ to span $\R^D$).
For every $N\in\N$, a spin configuration takes the form $\sigma = (\sigma_{1},\dots , \sigma_{N}) \in [-1,+1]^N$, where each spin~$\sigma_n$ is independently drawn from $P_\s$ if $n\in I_{N,\s}$.
In other words, denoting by $P_N$ the distribution of $\sigma$, we have
\begin{align*}
    \d P_N(\sigma) = \otimes_{\s\in\sS}\otimes_{i \in I_{N,\s}} \d P_\s(\sigma_i).
\end{align*}

For two spin configurations $\sigma,\sigma'$ of size $N$ and $\s\in\sS$, we consider the overlap of the $\s$-species:
\begin{align*}
    R_{N,\s}(\sigma,\sigma') = \tfrac{1}{N}\sigma_{I_{N,\s}} \cdot\sigma'_{I_{N,\s}} \in [-1,+1],
\end{align*}
where $\sigma_{I_{N,\s}} = (\sigma_i)_{i\in I_{N,\s}}$ is a vector in $\R^{I_{N,\s}}$ and similarly for $\sigma'_{I_{N,\s}}$. Hence, the spin overlap is $\R^\sS$-valued:
\begin{align*}
    R_N(\sigma,\sigma') = \Ll(R_{N,\s}(\sigma,\sigma')\Rr)_{\s\in\sS}.
\end{align*}

Let $\xi:\R^\sS\to\R$ be a smooth function and assume the existence of a centered Gaussian process $(H_N(\sigma))_{\sigma\in [-1,+1]^N}$ with covariance
\begin{align*}
    \E\Ll[ H_N(\sigma)H_N(\sigma')\Rr] = N\xi\Ll(R_N(\sigma,\sigma')\Rr).
\end{align*}
Again, $\xi$ is not assumed to be convex.

Fix any $\s\in\sS$.
Let $L_\s\in\N$ and let $(\chi_{\s,i})_{i\in I_{N,\s}}$ be i.i.d.\ $\R^{L_\s}$-valued random vectors, which are independent from other randomness. 
Let $d_\s\in\N$ and let $h_\s:\R\times \R^{L_\s} \to \R^{d_\s}$ be a bounded measurable function.
For each $i\in I_{N,\s}$, we view $h_\s(\sigma_i,\chi_{\s,i})$ as a generalized spin. We define the $\s$-species mean magnetization and the total mean magnetization as, respectively,
\begin{align*}
    m_{N,\s} = \frac{1}{N}\sum_{i\in I_{N,\s}} h_\s(\sigma_i,\chi_{\s,i})\quad\text{and}\quad m_N=\Ll(m_{N,\s}\Rr)_{\s\in\sS},
\end{align*}
which are $\R^{d_\s}$-valued and $\prod_{\s\in\sS}\R^{d_\s}$-valued.
Let $G:\prod_{\s\in\sS}\R^{d_\s}\to\R$ be a continuous function and we add to the system the Curie--Weiss-type interaction $NG\Ll(m_N\Rr)$.

Recall that in the setting of the vector spin glass, $\mcl Q$ is the collection of increasing càdlàg paths $q:[0,1) \to \S^D_+$. Here, we take $\mcl Q(1)$ to be $\mcl Q$ with $D=1$. In other words, $\mcl Q(1)$ consists of all increasing càdlàg paths from $[0,1)$ to $\R_+$. Here, $1$ indicates that the paths are one-dimensional. Then, for $\sfp\in[0,\infty]$, we write $\mcl Q_\sfp(1) = \mcl Q(1)\cap L^\sfp([0,1);\R)$. For each $q\in\mcl Q_\infty(1)$, we set $q(1) = \lim_{s\uparrow1}q(s)$. The collection of paths relevant in the multi-species setting is
\begin{align*}
    \mcl Q^\sS_\sfp = \prod_{\s\in\sS}\mcl Q_\sfp(1).
\end{align*}
For each $q\in\mcl Q^\sS_\sfp$, we write $q=(q_\s)_{\s\in\sS}$ with each $q_\s \in \mcl Q_\sfp(1)$.

For $\s\in\sS$ and $q_\s\in \mcl Q_\infty(1)$, let $(w^{q_\s}_i)$ be i.i.d.\ centered Gaussian fields where the covariance of each $(w^{q_\s}_i(\alpha))_{\alpha\in\supp\fR}$ is given similarly as in~\eqref{e.Ew^qw^q=}.  
For $N\in \N$, $t \in\R_+$, $q\in\mcl Q_\infty^\sS$, 
we consider the Hamiltonian
\begin{align*}
    H^{G}_{N;t,q}(\sigma,\alpha) &= \sqrt{2t}H_N(\sigma) - tN\xi\Ll(R_N(\sigma,\sigma)\Rr)
    \\
    &+ \sqrt{2}\sum_{\s\in\sS}\sum_{i\in I_{N,\s}} w^{q_\s}_i(\alpha) \sigma_i - Nq(1)\cdot R_N(\sigma,\sigma) 
    \\
    &+ NG\Ll(m_N\Rr). 
\end{align*}
Then, we consider the enriched free energy defined by
\begin{align}\label{e.F_N(t,q;s,x)=-ms}
    \bar F^G_N(t,q) &= -\frac{1}{N}\E\log \iint \exp\Ll(H^{G}_{N;t,q}(\sigma,\alpha)\Rr) \d P_N(\sigma)\d\mathfrak{R}(\alpha)
\end{align}
where $\E$ averages all the randomness.

For $q=(q_\s)_{\s\in\sS}\in\mcl Q_\infty^\sS$, $x=(x_\s)_{\s\in\sS}\in\prod_{\s\in\sS}\R^{d_\s}$, and $\lambda\in\blacktriangle_\infty$, we set
\begin{align*}
    &\psi_\s(q_\s;x_\s) 
    \\
    &= -\E\log\iint\exp\Ll(\sqrt{2}w^{q_\s}(\alpha)\tau-q_\s(1)\tau^2+x_\s\cdot h_\s(\tau,\chi_{\s,1})\Rr) \d P_\s(\tau)\d\mathfrak{R}(\alpha)
\end{align*}
and
\begin{align*}
    \psi_\lambda(q;x) = \sum_{\s\in\sS}\lambda_\s \psi_\s(q_\s;x_\s).
\end{align*}
Given $\xi$, we define
\begin{align*}
    \theta(a) = a\cdot\nabla\xi(a)-\xi(a),\quad\forall a \in\R^\sS.
\end{align*}



\subsection{Adapted main results}

Theorems~\ref{t.free_energy}, \ref{t.ldp}, and~\ref{t.rs} can be adapted as follows in this setting. We define $t_\star$ by

\begin{align}\label{e.t_star=-ms}
     t^{-1}_\star = \sup_x \Ll\|\nabla_q \psi_{\lambda_\infty}(\scdot;x)\Rr\|_{\mathrm{Lip}(\mcl Q^\sS_\infty,\mcl Q^\sS_\infty)}\Ll\|\nabla \xi\lfloor_{[-1,1]^\sS}\Rr\|_{\mathrm{Lip}([-1,1]^\sS;\R^\sS)}.
\end{align}

\begin{theorem}[Limit free energy] \label{t.free_energy-ms}
     Assume~\eqref{e.limlambda_N=lambda_infty}. For $x \in \prod_{\s\in\sS}\R^{d_\s}$, there is a twice Fréchet differentiable function $f(\cdot,\scdot;x) : [0,t_\star) \times \mcl{Q}_\infty^{\sS} \to \R$ that satisfies
     \begin{e*}
         \begin{cases}
              \partial_t f(t,q;x) - \int \xi(\nabla_q f(t,q;x)) = 0 &\quad\text{for } (t,q)\in (0,t_\star) \times  \mcl{Q}_\infty^{\sS}, \\ f(0,q;x)  = \psi_{\lambda_\infty}(q;x) &\quad\text{for }  q \in \mcl{Q}_\infty^{\sS}.
         \end{cases}
     \end{e*}
     For every $(t,q) \in [0,t_\star) \times \mcl{Q}_\infty^{\sS}$ and every continuous function $G:\prod_{\s\in\sS}\R^{d_\s}\to\R$ we have 
     \begin{e*}
        \lim_{N \to +\infty} \bar F^G_N(t,q) = \inf_m \sup_{x} \left\{ f(t,q;x) + m \cdot x - G(m) \right\}.
    \end{e*}
    Moreover, for every $x \in \prod_{\s\in\sS}\R^{d_\s}$, there is a unique solution $p^x\in\mcl Q_\infty^\sS$ of the fixed point equation
\begin{e*} 
        p = \nabla_q \psi_{\lambda_\infty}(q+t \nabla \xi(p);x) \quad\text{for $p\in \mcl Q_\infty^\sS$}.
\end{e*} 
We have,
\begin{align*}
    f(t,q;x ) &= \psi_{\lambda_\infty}(q +t \nabla \xi(p^x);x)
    - t \int_0^1 \theta(p^x(u)) \d u, \\
    \nabla_qf(t,q;x) &= p^x.
\end{align*}
\end{theorem}
The Gibbs measure associated to $H^{G}_{N;t,q}$ is the (random) probability measure defined on $R^N \times \supp \mfk R$  by
\begin{align*}
    \la\cdot\ra_{N;t,q}^G \quad\propto\quad \exp\Ll(H^G_{N;t,q}(\sigma,\alpha)\Rr)\d P_N(\sigma)\mathfrak{R}(\alpha).
\end{align*}
We denote by $\Omega$ the underlying probability space for $H^G_{N;t,q}$ and $\mfk R$, and write $\la\cdot\ra_{N;t,q}^{G,\omega}$ as a realization of $\la\cdot\ra_{N;t,q}^G$ for any $\omega\in\Omega$.

\begin{theorem}[Large deviations of $m_N$] \label{t.ldp-ms}
    Assume~\eqref{e.limlambda_N=lambda_infty}. Let $t < t_\star$ and $q \in \mcl Q_\infty^\sS$. There is a full measure subset $\Omega' \subset \Omega$ such that the following holds. For every $\omega\in\Omega'$ and every continuous function $G:\prod_{\s\in\sS}\R^{d_\s}\to\R$, the random variable $m_N$ satisfies a large deviation principle under $\la\cdot\ra_{N;t,q}^{G,\omega}$ with rate function $I^G_{t,q}$ defined by
    \begin{e*}
        I_{t,q}^G(m) =  -G(m) + f^*(t,q;m)  + \sup_{m'} \left\{ G(m') - f^*(t,q;m') \right\} ,\  \forall m \,  \in \prod_{\s\in\sS}\R^{d_\s},
    \end{e*}
    where $f^*(t,q;m)  = \sup_x \left\{ x \cdot m + f(t,q;x)\right\}$ with $f$ given in Theorem~\ref{t.free_energy-ms}.
\end{theorem}

Given $y\in\R^\sS_+$, we define $\hat y$ similar as in~\eqref{e.haty=} to be a constant path. 
For $y=(y_\s)_{\s\in\sS}\in\R^\sS_+$, $x=(x_\s)_{\s\in\sS}\in\prod_{\s\in\sS}\R^{d_\s}$, and $\lambda\in\blacktriangle_\infty$, we set
\begin{align*}
    &\phi_\s(y_\s;x_\s) 
    = -\E\log\int\exp\Ll(\sqrt{2 y_\s}\eta_\s\tau-y_\s\tau^2+x_\s\cdot h_\s(\tau,\chi_{\s,1})\Rr) \d P_\s(\tau),
\end{align*}
where $\eta_\s$ is an independent standard Gaussian random variable, and we also set
\begin{align*}
    \phi_\lambda(y;x) = \sum_{\s\in\sS}\lambda_\s \phi_\s(y_\s;x_\s).
\end{align*}
By a similar argument as in Remark~\ref{r.Gauss_ext_field}, we have $\phi_\lambda(y;x) = \psi_\lambda(\hat y;x)$.

\begin{theorem}[Replica symmetric form]\label{t.rs-ms}
Assume~\eqref{e.limlambda_N=lambda_infty}. Let $x\in\prod_{\s\in\sS}\R^{d_\s}$. Define $g(\scdot,\scdot;x):[0,t_\star)\times \R^\sS_+\to\R$ by
    \begin{e*}
        g(t,y;x)=f(t,\hat y;x),
    \end{e*}
with $f$ given in Theorem~\ref{t.free_energy-ms}.
    The function $g(\scdot,\scdot;x)$ is differentiable on $[0,t_\star)\times \R^\sS_+$ and satisfies
    \begin{align*}
        \begin{cases}
            \partial_t g(t,y;x) - \xi\big(\nabla_y g(t,y;x)\big)=0 &\quad \text{ for } (t,y)\in(0,t_\star)\times \R^\sS_+, \\
            g(0,y;x) = \phi_{\lambda_\infty}(y;x) &\quad \text{ for } y \in \R^\sS_+. 
        \end{cases}
    \end{align*}
    Moreover, at every $(t,y)\in (0,t_\star)\times \R^\sS_+$, there is a unique solution $z\in\R^\sS_+$ to the fixed point equation
    \begin{align*}
        z' = \nabla_y \phi_{\lambda_\infty}(y+t\nabla\xi(z');x) \quad\text{for }z'\in\R^\sS_+,
    \end{align*}
    and we have
    \begin{e*}
        g(t,y;x) = \phi_{\lambda_\infty}(y+t\nabla\xi(z);x) - t \theta(z)\quad\text{and}\quad \nabla_y g(t,y;x) = z.
    \end{e*}
\end{theorem}

\subsection{Recipe for adaptation}

Instead of rewriting everything in detail, we want to avoid repetition and give a recipe for adapting existing results to the multi-species setting.

We mention the difference in initial conditions.
In the vector spin glass setting, we simply have $\bar F^{x \mapsto x \cdot m}_N(0,q)=\psi(q;x)$ (see the line above~\eqref{e.psi(q;x)=}). Here, we have (see~\cite[Lemma~4.11]{multi-sp}) $\bar F^{x \mapsto x \cdot m}_N(0,q)=\psi_{\lambda_N}(q;x)$ and thus, under the assumption that $\lambda_N$ converges to $\lambda_\infty$,
\begin{e*}
    \lim_{N\to\infty} \bar F^{x \mapsto x \cdot m}_N(0,q) = \psi_{\lambda_\infty}(q;x).
\end{e*}
In the current setting, $\psi_\s$ has the same form as $\psi$ in~\eqref{e.psi(q;x)=} in the vector spin glass. So, each $\psi_\s$ satisfies the properties in Proposition~\ref{p.psi_smooth} with $\mcl Q_\sfp$ therein replaced by $\mcl Q_\sfp(1)$.

To transfer these properties to $\psi_{\lambda_\infty}$, we first recall the norms on $\mcl Q^\sS_\sfp$. We view $\mcl Q^\sS_\sfp$ as embedded into $L^\sfp([0,1);\R^\sS)$, equipped with the norm 
\begin{align*}
    |\kappa|_{L^\sfp} =  \Big(\int_0^1\Big(\sum_{\s\in\sS}\kappa_\s(r)^2\Big)^\frac{\sfp}{2}\d r\Big)^\frac{1}{\sfp}.
\end{align*} 
Observe that at $\sfp=2$, we have $|\kappa|_{L^2}^2=\sum_{\s\in\sS}|\kappa_\s|_{L^2}^2$.

One can verify that Proposition~\ref{p.f=UZ}, Corollary~\ref{c.limF_N=f}, and Lemma~\ref{l.varphi} hold with $t_\star$ defined in~\eqref{e.t_star=-ms}. Indeed, checking the proofs, one can see that the definition of $t_\star$ (previously in~\eqref{e.t_star=}) are used in~\eqref{e.|nabla^2xi(nablapsi)nabla^2psi<}, the line below \eqref{e.q=X(t,Z(t,q))=Z...}, \eqref{e.pf.l.unique crit point}, 
and \eqref{e.t_star_used_here_2}. Then, one can verify that the definition in~\eqref{e.t_star=-ms} still works in all these cases. To restate and reprove these four propositions, one needs to replace $\mcl Q_\sfp$, $\psi$ (see~\eqref{e.psi(q;x)=}), $\S^D_+$, and $\R^d$ therein by $\mcl Q^\sS_\sfp$, $\psi_{\lambda_\infty}$, $\R^\sS_+$, and $\prod_{\s\in\sS}\R^{d_\s}$, respectively. Cited results from~\cite{chenmourrat2023cavity} for vector spin glasses can be replaced by counterparts in~\cite{multi-sp} for multi-species spin glasses. 

Once those three results are proven in the setting of multi-species models, one can straightforwardly use the same arguments as in the rest of Section~\ref{s.proof_main_results} to finish the proofs of Theorems~\ref{t.free_energy-ms}, \ref{t.ldp-ms}, and~\ref{t.rs-ms}.

\bigskip

\noindent\textbf{Competing interests.}
The authors have no competing interests or other interests that might be perceived to influence the results or discussion reported in this paper.

\medskip

\noindent\textbf{Author contributions.}
Both HBC and VI made substantial contributions to the conception and design of the work, drafted the manuscript, approved the version to be published, and agree to be accountable for all aspects of the work.

\medskip

\noindent\textbf{Data availability.}
Not applicable.

\medskip

\noindent\textbf{Funding.}
HBC acknowledges funding from the NYU Shanghai Start-Up Fund and support from the NYU--ECNU Institute of Mathematical Sciences at NYU Shanghai.

\small
\bibliographystyle{plain}
\bibliography{ref}

\end{document}